\documentclass[11pt]{article}
\usepackage{amsmath,amsfonts,amsthm}
\usepackage{color}
\usepackage{hyperref}
\theoremstyle{definition}
\newtheorem{assumption}{Assumption}
\usepackage{amssymb,fullpage}
\usepackage{enumerate}
\usepackage{graphicx}
\usepackage{subfig}
\usepackage{makeidx}
\usepackage[utf8]{inputenc}
\usepackage{wrapfig}
\usepackage{mathrsfs}
\usepackage{latexsym}
\usepackage{amsbsy}
\usepackage{color}
\usepackage{bbm}
\usepackage{mathrsfs}
\usepackage{todonotes}
\usepackage{url}

\usepackage{wasysym}
\def\email#1{\it #1\par}

\providecommand{\otherindexspace}[1]{}

\newtheorem{theorem}{Theorem}[section]
\newtheorem{lemma}[theorem]{Lemma}
\newtheorem{proposition}[theorem]{Proposition}
\newtheorem{remark}[theorem]{Remark}
\newtheorem{definition}[theorem]{Definition}

\newtheorem{corollary}[theorem]{Corollary}

\DeclareMathOperator*{\argmax}{arg\,max}

\definecolor{falured}{rgb}{0.5, 0.09, 0.09}
\hypersetup{colorlinks=true, linkcolor={blue}, citecolor={falured}}

\usepackage{comment}

\usepackage{dsfont}

\usepackage[backend=biber, style=numeric, sorting=nyt, maxnames=50]{biblatex}
\numberwithin{equation}{section}

\usepackage[ruled,vlined, linesnumbered]{algorithm2e}

\title{Linear programming fictitious play algorithm for mean field games with optimal stopping and absorption\footnote{Peter Tankov gratefully acknowledges financial support from the ANR (project EcoREES ANR-19-CE05-0042) and from the FIME Research Initiative. }}
\author{ Roxana Dumitrescu \thanks{Department of Mathematics, King's College London, Strand, London, WC2R 2LS, United Kingdom, Email: \email roxana.dumitrescu@kcl.ac.uk} \and Marcos Leutscher  \thanks{CREST, ENSAE, Institut Polytechnique de Paris, 5 avenue Henry Le Chatelier, 91120 Palaiseau, France, Email: \email marcos.leutscherdelasnieves@ensae.fr} \and Peter Tankov \thanks{CREST, ENSAE, Institut Polytechnique de Paris, 5 avenue Henry Le Chatelier, 91120 Palaiseau, France, Email: \email peter.tankov@ensae.fr}}

\date{}

\begin{document}
\maketitle

\begin{abstract}
{
We develop the \textit{fictitious play} algorithm in the context of the {linear programming approach} for mean field games of {optimal stopping} and mean field games with {regular control and absorption}. This algorithm allows to approximate the mean field game population dynamics without computing the value function by solving linear programming problems associated with the distributions of the players still in the game and their stopping times/controls. We show the convergence of the algorithm using the topology of convergence in measure in the space of subprobability measures, which is needed to deal with the lack of continuity of the flows of measures. Numerical examples are provided to illustrate the convergence of the algorithm.
}

\end{abstract}

 \vspace{10pt}

\textbf{AMS: 91A55, 91A13, 60G40}

\section{Introduction}

The goal of this paper is to develop a numerical algorithm for computing Nash equilibria in Mean-field games of optimal stopping and Mean-field games with regular control and absorption.

Mean-Field Games (MFGs) are useful for approximating $N$-player Nash equilibria, which are rarely tractable. They have been introduced at about the same time by Lasry and Lions \cite{lasry2006a, lasry2006b, lasry2007} and Huang, Malhamé and Caines \cite{huang2006} as limit version of games with a large number of agents, symmetric interactions and negligible individual influence of each player on the others.  In the literature, several approaches have been developed to prove existence of an MFG Nash equilibrium. The analytic approach, introduced by Lasry and Lions and Huang, Malhamé and Caines, boils down to solving a coupled system of nonlinear partial differential equations: a Hamilton-Jacobi-Bellman equation (backward in time) satisfied by the value function of the representative agent and a Fokker-Planck-Kolmogorov equation (forward in time) describing the evolution of the density of agents when the optimal control is used. The probabilistic approach, introduced by Carmona and Delarue, based on the stochastic maximum principle, reduces the problem to a system of coupled forward-backward stochastic differential equations of McKean-Vlasov type.  Finally, the compactification methods consist in relaxing the optimization problems, and often allow to prove existence under weaker assumptions than the first two approaches. We refer here to the controlled martingale problem approach, introduced in the MFG setting by Lacker \cite{lacker2015} and the linear programming approach, developed for MFGs of optimal stopping in \cite{bdt2020} and extended to a more general framework in \cite{dlt2021}. In the single-agent case the linear programming approach is described in various papers, e.g. \cite{stockbridge1998, stockbridge2002, stockbridge2017}. A similar approach has been introduced in the works \cite{lewis1980} (deterministic case) and \cite{fleming1988} (stochastic case) and in the context of \textit{Aubry-Mather theory} in the deterministic and stochastic cases (see e.g. \cite{mather1991, mane1996}, \cite{gomes2010} and \cite{gomes2002}). In the  mean-field game setting, Aubry-Mather theory has been applied in the recent paper \cite{gomes2021} for studying a price formation MFG model.

Fast numerical methods for computing Nash equilibria are very important for applications. Several algorithms exist in the literature for the case of regular control without absorption. These use either analytic or probabilistic approach and iterate over the {value function} and the {mass distribution} (see e.g.\cite{achdou2010, Achdou2020, Chassagneux2019, angiuli2019}). Another method, based on the fictitious play algorithm which goes back to Brown \cite{brown1951} in the classical setting of game theory, has been introduced in the Mean-field framework in \cite{cardaliaguet2017b}. The fictitious play algorithm is a {learning procedure}, very natural in this setting. Due to the complexity of the game, it is unrealistic to assume that the agents can actually compute the equilibrium configuration. Such a configuration may only arise if the players {learn} how to play the game. The fictitious play for MFGs has been applied in the settings of potential MFGs with regular controls (\cite{cardaliaguet2017b}) and first order MFGs with regular controls (\cite{hadikhanloo2017, hadikhanloo2019, Elie2020}). A continuous time version of the fictitious play has been studied in \cite{perrin2020} for finite MFGs with common noise.

The compactification methods based on the {controlled martingale problem approach}, although they simplify the proofs of existence, are very abstract and cannot be exploited for the development of numerical algorithms. On the contrary, the {linear programming formulation} seems appropriate to build numerical schemes, see e.g. \cite{Mendiondo1998, stockbridge2002}. 

In this paper we develop and study the \textit{linear programming fictitious play algorithm} (LPFP) for {mean-field games of optimal stopping} and {regular control with absorption}, in the case of  second-order possibly non-potential games, under general assumptions on the coefficients and a non-strict monotonicity condition on the reward function with respect to the measure. 

The LPFP  algorithm  starts with an initial guess of the equilibrium and iterates the following  two steps:
\begin{enumerate}[(i)]
\item Compute the best response corresponding to the guess by solving a linear program. 
\item Update the guess via a convex combination (with well chosen parameters) of the previous guess and  the best response computed in Step (i). 
\end{enumerate}
While the linear program is infinite dimensional, one can approximate it by a finite dimensional linear program (see e.g. \cite{Mendiondo1998, stockbridge2002}), for which fast and accurate algorithms exist in most computing environments. We provide several numerical examples, which illustrate the convergence of the algorithm.

We emphasize that unlike other algorithms which iterate  over both, the {value function} and the {distribution of agents}, the LPFP algorithm iterates only over the distribution of agents, which is the main object of interest in the mean-field game setting. 

Very few papers present numerical algorithms for Mean-field games of optimal stopping and discuss their convergence. \cite{bdt2020} prove the convergence of fictitious play for potential games and \cite{bertucci2020} studies the Uzawa algorithm for the (possibly non-potential) MFG system introduced in \cite{bertucci2017} in the stationary case, under the assumption of strict monotonicity of the reward map. This LPFP algorithm has already been used for applications to water management and electricity markets in \cite{bdt2022} and \cite{adt2021}, respectively, but no theoretical convergence results have been provided. In the recent paper \cite{dianetti2022}, the authors study the class of submodular MFGs. In particular, in the case of optimal stopping, they prove the existence of equilibria using Tarski's fixed point theorem and provide an approximation of the minimal equilibria (with respect to a specific order structure) by starting with the minimal measure flow and iterating the minimal best response. The results hold under the opposite inequality to the Lasry-Lions monotonicity condition used in this paper. The theory on mean-field games with regular control and absorption has received a lot of interest recently, being developed in \cite{campi2018, campi2021, burzoni2021}, and using the linear programming approach in \cite{dlt2021}. To the best of our knowledge, no algorithm has been proposed so far in this setting and one of our goals is to fill this gap.

One of the principal difficulties of MFGs with stopping/absorption is that the standard weak convergence topology for measures cannot be used due to the possible lack of regularity of the flow of measures. In this paper, we solve this problem and provide general convergence results for LPFP, by finding an appropriate topology, i.e. the topology of the convergence in measure in infinite-dimensional spaces and providing appropriate estimates using well chosen metrics.
Furthermore, the use of this topology allows us to prove the existence of an equilibria in a much more general framework compared to \cite{bdt2020} and \cite{dlt2021}.

The paper is organized as follows. In Section 2, we introduce the LPFP algorithm for MFGs of optimal stopping. We study the compactness under the convergence in measure topology of the set of admissible measures and provide several key estimates, which are used to prove the convergence of the algorithm. Numerical illustrations are provided. In Section 3, we propose an LPFP algorithm in the case of regular control with absorption, and show its convergence by using the tools developed in Section 2. In the Appendix we give some technical results and some examples of sufficient conditions under which the main results hold.

\vspace{10mm}

\paragraph{Notation.} For a topological space $(E, \tau)$ we denote by $\mathcal{B}(E)$ the Borel $\sigma$-algebra, by $\mathcal{M}^s(E)$ the set of Borel finite signed measures on $E$, by $\mathcal{M}(E)$ the set of Borel finite positive measures on $E$, by $\mathcal{P}^{sub}(E)$ the set of Borel subprobability measures on $E$ and by $\mathcal{P}(E)$ the set of Borel probability measures on $E$. We denote by $M(E)$ the set of Borel measurable functions from $E$ to $\mathbb R$, by $M_b(E)$ the subset of Borel measurable and bounded functions, by $C(E)$ the subset of continuous functions, and by $C_b(E)$ the subset of continuous and bounded functions. The set $M_b(E)$ is endowed with the supremum norm $\|\varphi\|_\infty=\sup_{x\in E}|\varphi(x)|$. If $(E, d)$ is a metric space and $p\geq 1$, we denote by $\mathcal{M}^s_p(E)$ (respectively $\mathcal{M}_p(E)$, $\mathcal{P}_p^{sub}(E)$ and $\mathcal{P}_p(E)$) the set of $\mu\in \mathcal{M}^s(E)$ (respectively $\mathcal{M}(E)$, $\mathcal{P}^{sub}(E)$ and $\mathcal{P}(E)$) such that there exists a point $x_0\in E$ so that $\int_E d(x, x_0)^p|\mu|(dx)<\infty$, where $|\mu|$ is the total variation measure of $\mu$.

Let $T>0$ be a terminal time horizon, $\mathcal{O}$ be an open subset of $\mathbb R$ with closure $\bar{\mathcal{O}}$ and $A$ be a compact subset of $\mathbb{R}$. We denote by $C_b^{1, 2}([0, T]\times \bar{\mathcal{O}})$ the set of functions $u\in C_b([0, T]\times \bar{\mathcal{O}})$ such that $\partial_t u, \partial_x u, \partial_{xx}u \in C_b([0, T]\times \bar{\mathcal{O}})$. We denote by $\mathbb R_+$ the set $[0, +\infty[$. For a given process $(Y_t)_t$ and a Borel subset $B$ of $\mathbb R$, we define the random time $\tau_{B}^Y(\omega):=\inf\{t\geq 0: Y_t(\omega)\notin B\}$, with the convention $\inf\emptyset =\infty$. 

Since we will deal with different topologies and distances throughout the paper, we list them here,  making reference to the places where they are introduced: the topologies $\tau_0$ (weak convergence), $\tau_p$ (weak convergence with $p$-growth), $\bar \tau$ (stable convergence), $\bar \tau_p$ (stable convergence with $p$-growth) and the metric $d_{\text{BL}}$ are introduced in Appendix \ref{app tau p}. The topology $\tilde\tau_p$ of the convergence in measure for flows of subprobabilities in $\mathcal{P}_p^{sub}$ is defined just before Assumption \ref{assump existence OS}. The metric $W_1'$ on $\mathcal{P}_1^{sub}$ is defined just before Proposition \ref{distance tau_1} and the metrics $d_M$ and $\rho$ are defined in Proposition \ref{distance tau_1}.

\section{Linear programming algorithm for \textit{Optimal stopping MFGs}}\label{sec OS}

\subsection{Preliminaries and main result}

We describe here the LPFP algorithm for MFGs of optimal stopping, i.e. when players choose the time to exit the game. We present the definition of LP (Linear Programming) MFG Nash equilibrium in this setting and prove the convergence of the LPFP algorithm to the LP MFG Nash equilibrium.

\paragraph{Preliminaries.} Let $U$ be the set of flows of measures on $\bar{\mathcal{O}}$, $\left(m_{t}\right)_{t\in [0, T]}$, such that: for every $t \in[0, T]$, $m_{t}$ is a Borel finite signed measure on $\bar{\mathcal{O}}$, for every $B \in \mathcal{B}(\bar{\mathcal{O}})$, the mapping $t \mapsto m_{t}(B)$ is measurable, and $\int_{0}^{T} |m_{t}|(\bar{\mathcal{O}}) dt < \infty$, where $|m_{t}|$ is the total variation measure of $m_{t}$.

We define $\tilde U$ as the quotient space given by $U$ and the almost everywhere equivalence relation on $[0, T]$, that is, if $dt$-almost everywhere on $[0,T]$ the measures $m_{t}^1$ and $m_{t}^2$ coincide, the measure flows $(m^1_t)_{t\in [0, T]}$ and $(m^2_t)_{t\in [0, T]}$ are considered equivalent. $\tilde U$ endowed with the usual sum and scalar multiplication is a vector space, where the zero vector is given by the family of null measures $(\mathbf 0)_{t\in [0, T]}$. To each $(m_t)_{t\in [0, T]}\in \tilde U$ we associate a Borel finite signed measure on $[0, T]\times\bar{\mathcal{O}}$ defined by $m_t(dx)dt$ and we endow $\tilde U$ with the topology of weak convergence of the associated measures. For $p\geq 1$, define the subsets of $\tilde U$,
$$\tilde U_p:=\left\{m\in \tilde U: \int_0^T\int_{\bar{\mathcal{O}}}|x|^p|m|_t(dx)dt<\infty \right\},$$
endowed with the weak topology with respect to continuous functions with $p$-growth, denoted by $\tau_p$, of the associated measures (see Appendix \ref{app tau p}).
We denote by $V$ (resp. $V_p$) the set of measure flows $(m_t)_{t\in [0, T]}\in \tilde U$ (resp. $\tilde U_p$) such that $dt$-a.e. $m_t$ is a subprobability measure. We make the convention that when integrating a quantity with respect to $dt$, the version taken for $(m_t)_{t\in [0, T]}\in V$ inside the integral is such that, for each $t\in [0, T]$, $m_t$ is a subprobability measure. We note that $\tilde U$ is a Hausdorff locally convex topological vector space and $V_p$ is metrizable. We endow the set $\mathcal{P}_p([0, T]\times \bar{\mathcal{O}})$ with the topology $\tau_p$ and we will often work on the product space $\mathcal{P}_p([0, T]\times \bar{\mathcal{O}})\times V_p$ endowed with the product topology, which we will denote $\tau_p\otimes \tau_p$. Since this product space is metrizable, we will often work with sequences.  Finally, consider the set $M_p:=M([0, T];\mathcal{P}_p^{sub}(\bar{\mathcal{O}}))$ of Borel measurable functions from $[0, T]$ to $\mathcal{P}_p^{sub}(\bar{\mathcal{O}})$ identified a.e. on $[0, T]$. This set is endowed with the topology of convergence in measure (see Appendix \ref{app conv m}) which is denoted by $\tilde \tau_p$. Moreover, any $m\in V_p$ admits a representative in $M_p$. We can thus consider,  without loss of generality, the topology $\tilde \tau_p$ in $V_p$.\vspace{10pt} 

\noindent We are given constants $q> p\geq 1\vee r$, where $r\in [0, 2]$ and $q\geq 2$, and the following functions:
$$(b, \sigma):[0, T]\times \mathbb R\rightarrow \mathbb R,\quad f:[0, T]\times \bar{\mathcal{O}}\times \mathcal{P}_p^{sub}(\bar{\mathcal{O}})\rightarrow \mathbb R,\quad g:[0, T]\times \bar{\mathcal{O}}\times \mathcal{P}_p([0, T]\times \bar{\mathcal{O}})\rightarrow \mathbb R.$$
The sets $[0, T]$, $\mathbb R$ and $\bar{\mathcal{O}}$ are endowed with the usual topology and the sets $\mathcal{P}_p^{sub}(\bar{\mathcal{O}})$ and $\mathcal{P}_p([0, T]\times \bar{\mathcal{O}})$ are endowed with the topology $\tau_p$. Throughout the paper, we will adopt the bilinear form notation 
$$\langle f(m), m' \rangle:=\int_0^T\int_{\bar{ \mathcal{O}}}f(t, x, m_t)m_t'(dx)dt,\quad \langle g(\mu), \mu' \rangle:=\int_{[0, T]\times \bar{\mathcal{O}}}g(t, x, \mu)\mu'(dt, dx),
$$
where $(\mu, m), (\mu', m')\in \mathcal{P}_p([0, T]\times \bar{\mathcal{O}})\times V_p$.

In this section, we let the following assumptions hold true.

\begin{assumption}\label{assump existence OS}\leavevmode
\begin{enumerate}[(1)]
\item $m_0^*\in \mathcal{P}_q(\bar{\mathcal{O}})$.

\item The functions $(t, x)\mapsto b(t, x)$ and $(t, x)\mapsto \sigma(t, x)$ are jointly measurable and continuous in $x$ for each $t$. Moreover, there exists a constant $c_1>0$ such that for all $(t, x, y)\in [0, T]\times \mathbb R\times \mathbb R$,
$$|b(t, x)-b(t, y)|+|\sigma(t, x)-\sigma(t, y)|\leq c_1 |x-y|,\quad |b(t, x)|\leq c_1\left[1+|x|\right],\quad \sigma^2(t, x)\leq c_1\left[ 1+|x|^r\right].$$

\item The function $(t, x, m)\mapsto f(t, x, m)$ is jointly measurable and continuous in $(x, m)$ for each $t$. The function $g$ is jointly continuous. Moreover, there exists a constant $c_2>0$ such that for all $(t, x, m, \mu)\in [0, T]\times \bar{\mathcal{O}}\times \mathcal{P}_p^{sub}(\bar{\mathcal{O}})\times \mathcal{P}_p([0, T]\times \bar{\mathcal{O}})$,
$$|f(t, x, m)|\leq c_2\left[ 1+|x|^p + \int_{\bar{\mathcal{O}}}|z|^pm(dz) \right],\quad |g(t, x, \mu)|\leq c_2\left[ 1+|x|^p + \int_{[0, T]\times \bar{\mathcal{O}}}|z|^p\mu(ds, dz)\right].$$

\item One of the following statements is true:
\begin{enumerate}
\item \textit{Unattainable boundary}: $b$, $\sigma$ and $\mathcal{O}$ are such that, $\mathbb P \left(\tau_\mathcal{O}^{X}\geq T \right)=1$ where $X$ is the unique strong solution of
$$dX_t= b(t, X_t)dt + \sigma(t, X_t)dW_t,\quad \mathbb P \circ X_0^{-1}= m_0^*.$$
\item \textit{Attainable boundary}: $\mathcal{O}$ is a bounded open interval and for all $(t, x)\in [0, T]\times \mathbb R$, $\sigma^2(t, x)\geq c_\sigma$ for some $c_\sigma>0$. 
\end{enumerate}
\end{enumerate}
\end{assumption}

We now give the formulation of the \textit{linear programming Optimal Stopping MFG problem}, where the set of occupation measures induced by stopping times is replaced by the set of measures satisfying an infinite-dimensional linear constraint. This allows to compactify and convexify the optimization problem. As will be explained later in the paper, the flow of subprobability measures $m$ is not necessarily regular in time.  Because of this lack of regularity, MFGs of optimal stopping require a different treatment, and in particular, a different topology than classical MFGs (stochastic control without absorption). For the \textit{strong formulation of Optimal Stopping MFG}, we refer the reader to \cite{dlt2021}.

\paragraph{The relaxed MFG problem: the linear programming formulation.} 

\begin{definition}
Let $\mathcal{R}$ be the set of pairs $(\mu, m) \in \mathcal{P}_p([0, T]\times \bar{\mathcal{O}})\times V_p$, such that for all $u\in C_b^{1, 2}([0, T]\times \bar{\mathcal{O}})$,
\begin{equation*}
\int_{[0, T]\times \bar{\mathcal{O}}} u(t, x)\mu(dt, dx)= \int_\mathcal{O} u(0, x)m_0^*(dx) + \int_0^T \int_{\bar{\mathcal{O}}} \left(\partial_t u +\mathcal L u\right) (t, x)m_t(dx)dt,
\end{equation*}
where $\mathcal L u(t, x):= b(t, x)\partial_x u(t, x) + \frac{\sigma^2}{2}(t, x)\partial_{xx}u(t, x)$.
\end{definition}

\begin{definition}[\textit{LP formulation of the MFG problem}]
For $(\bar \mu, \bar m)\in \mathcal{P}_p([0, T]\times \bar{\mathcal{O}})\times V_p$, let $\Gamma[\bar \mu, \bar m]: \mathcal{P}_p([0, T]\times \bar{\mathcal{O}})\times V_p\rightarrow \mathbb R$ be the reward functional associated to $(\bar \mu, \bar m)$, defined by
$$\Gamma[\bar \mu, \bar m] (\mu, m)= \langle f(\bar m), m\rangle + \langle g(\bar \mu), \mu\rangle.$$
We say that $(\mu^\star, m^\star)\in \mathcal{P}_p([0, T]\times \bar{\mathcal{O}})\times V_p$ is an LP MFG Nash equilibrium if $(\mu^\star, m^\star)\in \mathcal{R}$ and for all $(\mu, m)\in \mathcal{R}$, $\Gamma[\mu^\star, m^\star] (\mu, m)\leq \Gamma[\mu^\star, m^\star] (\mu^\star, m^\star)$.
\end{definition}

For a finite number of iterations, the LPFP algorithm will produce what we call an $\varepsilon$-LP MFG Nash equilibrium. 

\begin{definition}
For a given $\varepsilon\geq 0$, we say that $(\mu^\star, m^\star)\in \mathcal{P}_p([0, T]\times \bar{\mathcal{O}})\times V_p$ is an $\varepsilon$-LP MFG Nash equilibrium if $(\mu^\star, m^\star)\in \mathcal{R}$ and for all $(\mu, m)\in \mathcal{R}$, $\Gamma[\mu^\star, m^\star] (\mu, m) - \varepsilon \leq \Gamma[\mu^\star, m^\star] (\mu^\star, m^\star)$.
\end{definition}

In the next paragraph, we present the main results of the paper, in particular the linear programming fictitious play algorithm and its convergence.

\paragraph{The Linear Programming Fictitious Play algorithm.} 
In order to show the convergence of the algorithm, we impose the following Assumption.
\begin{assumption}\label{main assump OS}
\begin{enumerate}[(1)]
\item For each $(\bar \mu, \bar m) \in \mathcal{R}$, there exists a unique maximizer of 
$\Gamma[\bar \mu, \bar m]$ on $\mathcal{R}$. 
\item The Lasry–Lions monotonicity condition holds: for all $(\mu, m)$ and $(\tilde \mu, \tilde m)$ in $\mathcal{P}_p([0, T]\times \bar{\mathcal{O}})\times V_p$,
\begin{equation*}
\langle f(m) - f(\tilde m),m-\tilde m\rangle 
+ \langle g(\mu)-g(\tilde \mu),\mu-\tilde \mu\rangle \leq 0.
\end{equation*}
\item There exist constants $c_f\geq 0$ and $c_g\geq 0$ such that for all $t\in [0, T]$, $x, x'\in \bar{\mathcal{O}}$, $m, m'\in \mathcal{P}^{sub}_p(\bar{\mathcal{O}})$, $\mu, \mu'\in \mathcal{P}_p([0, T]\times \bar{\mathcal{O}})$, 
$$\left|f(t, x, m) - f(t, x, m')\right|\leq c_f(1+|x|) \int_{\bar{\mathcal{O}}}(1+|z|^p)|m-m'|(dz),$$
$$|f(t, x, m) - f(t, x, m') - f(t, x', m) + f(t, x', m')|\leq c_f |x-x'|\int_{\bar{\mathcal{O}}}(1+|z|^p)|m-m'|(dz),$$ 
$$|g(t, x, \mu) - g(t, x, \mu')|\leq c_g (1+|x|) \int_{[0, T]\times \bar{\mathcal{O}}}(1+|z|^p)|\mu-\mu'|(ds, dz),$$
$$|g(t, x, \mu) - g(t, x, \mu') - g(t', x', \mu) + g(t', x', \mu')|\leq c_g (|t-t'|+|x-x'|)\int_{[0, T]\times \bar{\mathcal{O}}}(1+|z|^p)|\mu-\mu'|(ds, dz).$$
\end{enumerate}
\end{assumption}
\noindent For sufficient conditions on the coefficients under which the above assumptions hold, the reader is referred to Appendix \ref{app suff cond OS}.

\noindent Note that under Assumption \ref{main assump OS},  the function $\Theta:\mathcal{R}\rightarrow \mathcal{R}$ defined by
$$\Theta(\bar \mu, \bar m)=\argmax_{(\mu, m)\in \mathcal{R}}\Gamma[\bar \mu, \bar m](\mu, m), \quad (\bar \mu, \bar m)\in \mathcal{R}$$ is well defined. Furthermore, we can show the uniqueness of the LP MFG Nash equilibrium.

\begin{proposition}[\textit{Uniqueness of the equilibrium}]\label{uniq OS}
Under Assumptions \ref{assump existence OS} and \ref{main assump OS}, there exists at most one LP MFG Nash equilibrium.
\end{proposition}

\begin{proof}
Let $(\mu, m)\in \mathcal{R}$ and $(\mu', m')\in \mathcal{R}$ be two LP MFG Nash equilibria. Using the equilibrium property we get the following two inequalities:
$$\langle f(m), m-m' \rangle + \langle g(\mu), \mu- \mu' \rangle\geq 0,\quad \langle f(m'), m'-m \rangle + \langle g(\mu'), \mu'- \mu \rangle\geq 0.$$
Adding up these two inequalities we obtain $\langle f(m) - f(m'), m-m' \rangle + \langle g(\mu) - g(\mu'), \mu- \mu' \rangle\geq 0$. By the Lasry-Lions monotonicity condition we get the equality in the previous inequalities. Using that 
$$\Gamma[\mu, m](\mu, m)=\langle f(m), m \rangle + \langle g(\mu), \mu \rangle =\langle f(m), m' \rangle + \langle g(\mu), \mu' \rangle=\Gamma[\mu, m](\mu', m'),$$
we deduce by the uniqueness of the best response to $(\mu, m)$ that $(\mu, m)=(\mu', m')$.
\end{proof}

\noindent We propose the following algorithm for computing the LP MFG Nash equilibrium.\vspace{10pt}

\begin{algorithm}[H]\label{algo 1}
\SetAlgoLined
\KwData{A number of steps $N$ for the equilibrium approximation; a pair $(\bar \mu^{(0)}, \bar m^{(0)})\in \mathcal{R}$;}
\KwResult{Approximate LP MFG Nash equilibrium}
\For{$\ell=0, 1, \ldots, N-1$}{
Compute a linear programming best response $(\mu^{(\ell+1)}, m^{(\ell+1)})$ to $(\bar \mu^{(\ell)}, \bar m^{(\ell)})$ by solving the linear programming problem
$$\argmax_{(\mu, m)\in \mathcal{R}}\Gamma[\bar \mu^{(\ell)}, \bar m^{(\ell)}](\mu, m).$$
\\
Set $(\bar \mu^{(\ell+1)}, \bar m^{(\ell+1)}):=\frac{\ell}{\ell+1}(\bar \mu^{(\ell)}, \bar m^{(\ell)}) + \frac{1}{\ell+1}(\mu^{(\ell+1)},  m^{(\ell+1)})=\frac{1}{\ell + 1}\sum_{\nu=1}^{\ell+1}(\mu^{(\nu)},  m^{(\nu)})$ \\
}
\caption{LPFP algorithm (Optimal stopping MFGs)}
\end{algorithm}
\vspace{10pt}

We state now the main theorem. Its proof is provided in section \ref{sec main theo}.

\begin{theorem}[\textit{Convergence of the algorithm}]\label{main theo OS}
Let Assumptions \ref{assump existence OS} and \ref{main assump OS} hold true and consider the sequences $(\bar\mu^{(N)}, \bar m^{(N)})_{N\geq 1}$ and $(\mu^{(N)}, m^{(N)})_{N\geq 1}$ generated by the Algorithm \ref{algo 1}. Then both sequences converge in the product topology $\tau_p\otimes \tilde\tau_p$ to the unique LP MFG Nash equilibrium.
\end{theorem}

\paragraph{Numerical example.}

To illustrate the convergence of the algorithm, we solve numerically a simple MFG of optimal stopping. In this game, the state of the representative player belongs to the domain $[0,T]\times \mathcal O$ with $T=1$ and $\mathcal O=\mathbb R$ and is given by 
$$
X_t = X_0 + t + W_t,
$$
i.e. $b(t,x)=\sigma(t,x)=1$. 
The initial states of the players are distributed according to the law  $m_0^*=\mathcal{N}(0, 4)$. Before exiting the game, at each time $t$, the representative player receives an instantaneous reward given by 
$$
\int_\mathbb R (X_t - y)m_t(dy),
$$
where $m_t$ is the distribution of players still in the game at time $t$, and upon exiting the game at time $\tau$, the player receives the terminal reward
$$
\int_{[0,T]\times \mathbb R} (\tau- s)\mu(ds,dy), 
$$
where $\mu$ is the joint distribution of exit times and states of the players. The functions $f$ and $g$ are therefore defined as follows:
$$f(t, x, m):=\int_\mathbb R (x-y)m(dy),\quad g(t, x, \mu):=\int_{[0, T]\times \mathbb R}(t-s)\mu(ds, dy).$$

The representative player has an incentive to stay in the game if its state is higher than the average state of the other players who are still in the game. It is expected that the players starting with a low state value will exit the game immediately, while the players starting with a high state value will stay until the end of the game. 

To apply the LPFP algorithm, we discretize the linear program for the computation of the best response as in \cite{Mendiondo1998, stockbridge2002}. More precisely, we consider a time grid $t_i=i \Delta$ with $\Delta=\frac{T}{n_t}$, for $i\in \{0, 1, \ldots n_t\}$ and a state grid $x_{j+1}=x_j+\delta$, for $j\in \{0, 1, \ldots n_s-1\}$ with $x_0\in\mathbb R$ and $\delta>0$. We define
$$Lu(t, x)=\frac{\partial u}{\partial t}(t, x) + b(t, x)\frac{\partial u}{\partial x}(t, x) + \frac{\sigma^2}{2}(t, x)\frac{\partial^2 u}{\partial x^2}(t, x), \quad \forall u\in \mathcal{D}(L):=C^{1, 2}_b([0, T]\times \mathbb R).$$
We set $\mathcal{D}(\hat L)$ as the functions in $\mathcal{D}(L)=C^{1, 2}_b([0, T]\times \mathbb R)$ restricted to the time-state discretization grid. For $u\in \mathcal{D}(\hat L)$, we discretize the derivatives as follows
$$\hat L_t u (t_i, x_j)= \frac{1}{\Delta}[u(t_{i+1}, x_j)-u(t_i, x_j)],$$
$$\hat L_x^u u (t_i, x_j)=\frac{1}{\delta}\max(b(t_i, x_j), 0)[u(t_{i+1}, x_{j+1})-u(t_{i+1}, x_{j})],$$
$$\hat L_x^d u (t_{i}, x_j)=\frac{1}{\delta}\min(b(t_i, x_j), 0)[u(t_{i+1}, x_{j})-u(t_{i+1}, x_{j-1})],$$
$$\hat L_{xx} u (t_i, x_j)=\frac{1}{\delta^2}\frac{\sigma^2}{2}(t_i, x_j)[u(t_{i+1}, x_{j+1}) + u(t_{i+1}, x_{j-1})-2u(t_{i+1}, x_j)].$$
The discretized generator has the form:
$$\hat L u (t_i, x_j) = \hat L_t u (t_i, x_j) + \hat L_x^u u (t_i, x_j) + \hat L_x^d u (t_i, x_j) + \hat L_{xx} u (t_i, x_j).$$
The constraint reads as
$$
\sum_{i=0}^{n_t}\sum_{j=0}^{n_s} u(t_i, x_j)\mu(t_i, x_j) - \Delta \sum_{i=0}^{n_t-1}\sum_{j=1}^{n_s-1} \hat L u(t_i, x_j) m(t_i, x_j) = \sum_{j=1}^{n_s-1}u(t_0, x_j) m_0^*(x_j), 
$$
for $u\in \mathcal{D}(\hat L)$. The set $\mathcal{D}(\hat L)$ is equal to the linear span of the indicators functions 
$$\mathbf{1}_{\{(t_i, x_j)\}}, \quad i\in\{0, 1, \ldots, n_t\}, \; j\in\{0, 1, \ldots, n_s\}$$ 
on the time-state grid. By linearity, it suffices to evaluate the constraint on the set of indicator functions. We obtain a total number of $(n_t+1)\times (n_s+1)$ constraints. The discretized reward associated to a discrete mean-field term $(\bar\mu, \bar m)$ is given by
$$
\sum_{i=0}^{n_t}\sum_{j=0}^{n_s} g(t_i, x_j, \bar \mu)\mu(t_i, x_j) + \Delta\times \sum_{i=0}^{n_t-1}\sum_{j=1}^{n_s-1} f(t_i, x_j, \bar m(t_i, \cdot)) m(t_i, x_j).
$$
The generator obtained using these approximations is associated to the following Markov chain (see p. 328 in \cite{kushner2001}):
$$\mathbb P(Y_{t_{i+1}}=x_{j}|Y_{t_{i}}=x_{j})=1-\sigma^2(t_i, x_j)\frac{\Delta}{\delta^2} - |b(t_i, x_j)|\frac{\Delta}{\delta},$$
$$\mathbb P(Y_{t_{i+1}}=x_{j+1}|Y_{t_{i}}=x_{j})=\frac{\sigma^2}{2}(t_i, x_j)\frac{\Delta}{\delta^2} + b^+(t_i, x_j)\frac{\Delta}{\delta},$$
$$\mathbb P(Y_{t_{i+1}}=x_{j-1}|Y_{t_{i}}=x_{j})=\frac{\sigma^2}{2}(t_i, x_j)\frac{\Delta}{\delta^2} + b^-(t_i, x_j)\frac{\Delta}{\delta}.$$
For this to be well defined, we should have 
$$1-\sigma^2(t_i, x_j)\frac{\Delta}{\delta^2} - |b(t_i, x_j)|\frac{\Delta}{\delta}\geq 0,$$
meaning that we should have for all $i$ and $j$
$$\Delta\leq \frac{\delta^2}{\sigma^2(t_i, x_j) + \delta |b|(t_i, x_j)}.$$
The discretized constraint coincides with the constraint associated to the Markov chain $Y$ (with absorption on $\{x_0, x_{n_s}\}$). Convergence of this approximating procedure in the context of single-agent stochastic control is studied in \cite[Chapter 10]{kushner2001}.

These finite dimensional linear programs are solved using the Gurobi\footnote{\url{https://www.gurobi.com/}} solver in Python (Version 9.5.1). In order to evaluate the convergence of the algorithm, we compute the (discrete time-space) exploitability (term borrowed from \cite{perrin2020}) at each iteration: 
\begin{multline*}
\varepsilon_N=\Delta \sum_{i=0}^{n_t-1}\sum_{j=1}^{n_s-1}f(t_i, x_j, \bar m^{(N-1)}(t_i, \cdot))(m^{(N)}(t_i, x_j)-\bar m^{(N-1)}(t_i, x_j)) \\ 
+ \sum_{i=0}^{n_t}\sum_{j=0}^{n_s}g(t_i, x_j, \bar \mu^{(N-1)})(\mu^{(N)}(t_i, x_j)- \bar \mu^{(N-1)}(t_i, x_j)).
\end{multline*}

In Figure \ref{distribution OS} (left graph), we can observe the evolution of the distribution of the players $m_t$ over time and Figure \ref{distribution OS} (right) shows the exit distribution $\mu$ of the players. As expected, players starting at a low position exit  the game immediately and players starting at higher positions exit the game at later dates. Finally, Figure \ref{error_os} illustrates the convergence of the algorithm via a log-log plot (base $10$) of the exploitability. We can observe that the convergence is in $O(N^{-1})$.

\begin{figure}[h]
    \centering
    \subfloat{%
        \includegraphics[width=0.5\textwidth]{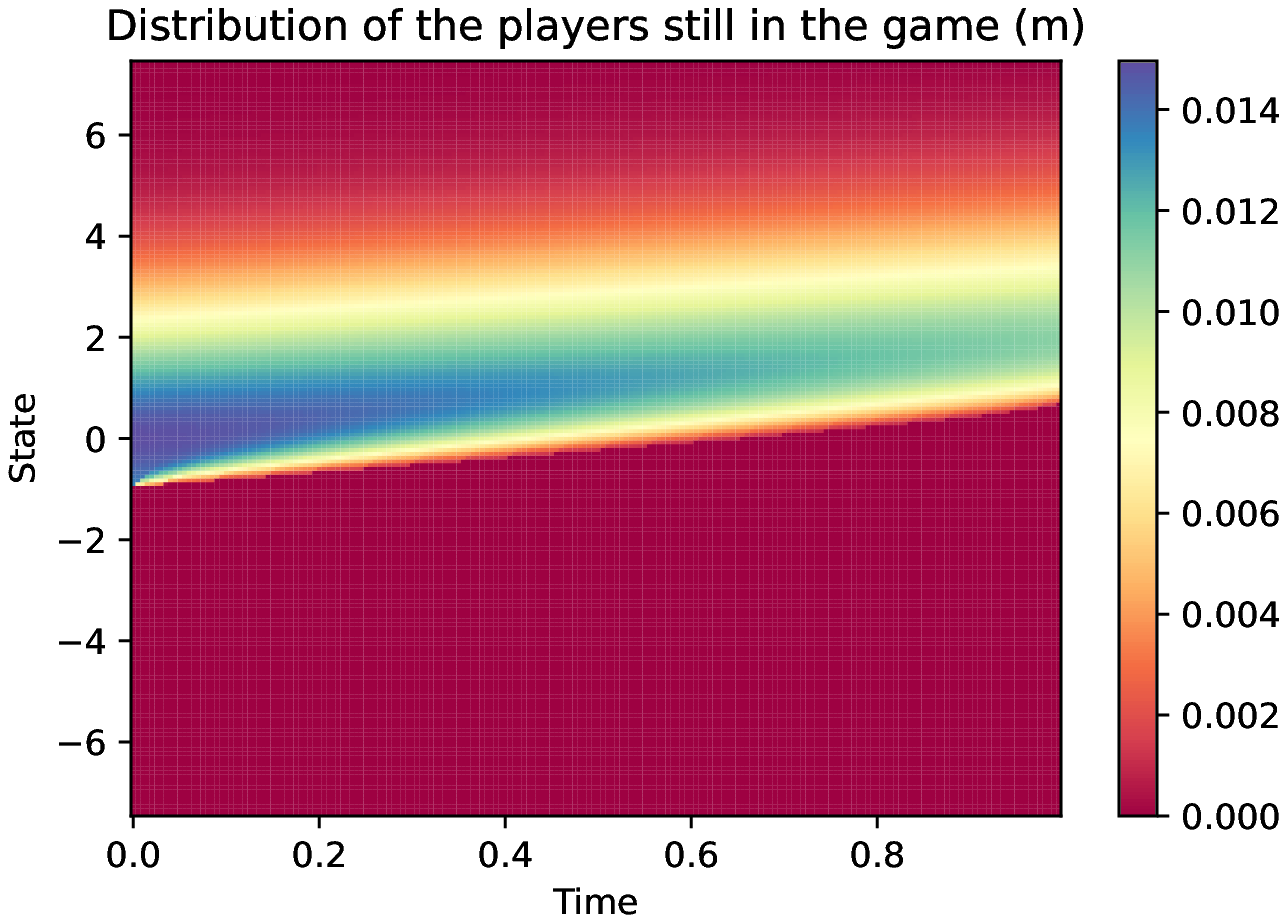}%
        \label{m_os}%
        }%
    \hfill%
    \subfloat{%
        \includegraphics[width=0.5\textwidth]{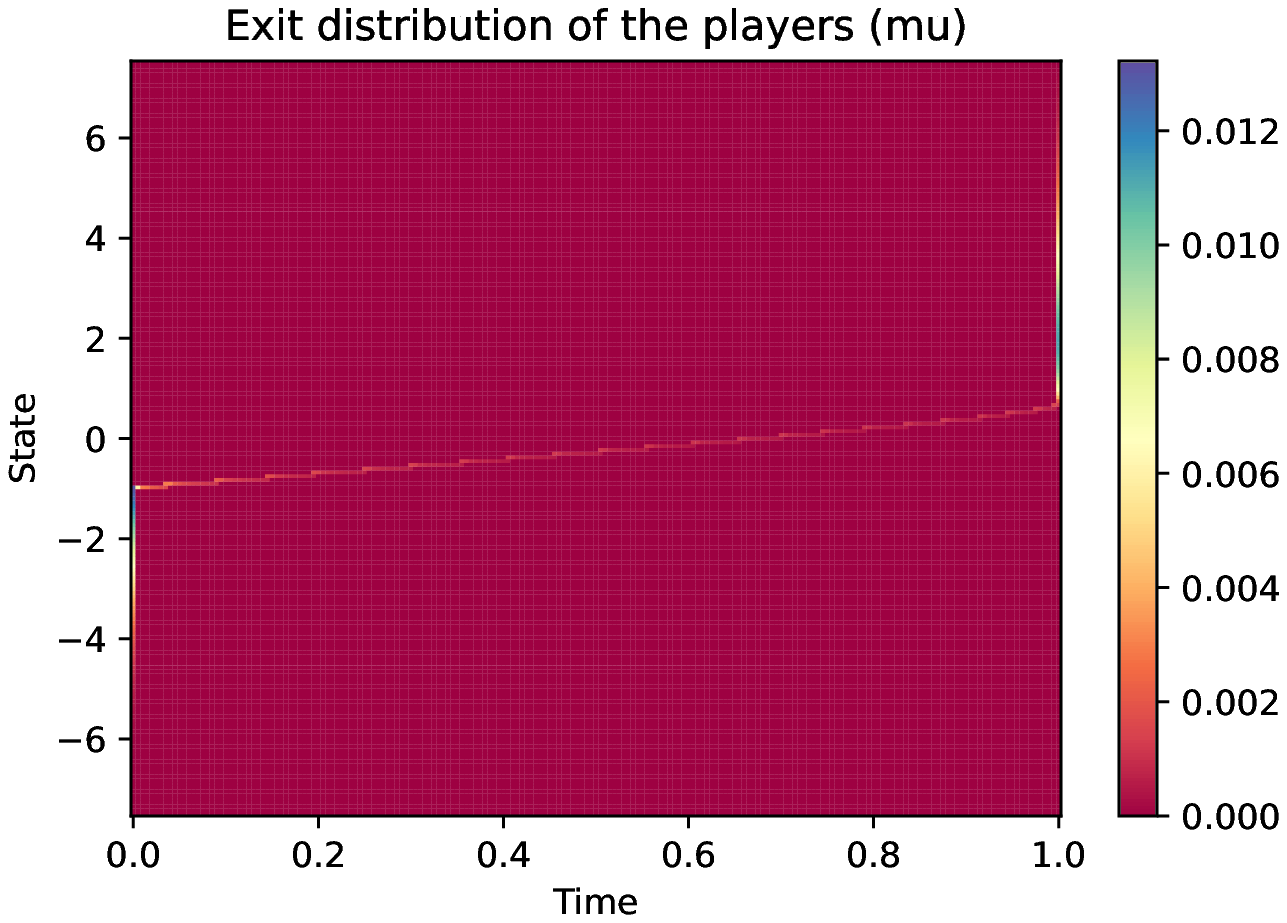}%
        \label{mu_os}%
        }%
    \caption{Equilibrium distributions at the final iteration. Left graph: $\bar m^{(N)}$, distribution of the players still in the game over time. Right graph: $\bar \mu^{(N)}$, exit distribution of the players.}
    \label{distribution OS}
\end{figure}

\begin{figure}[h]
    \centering
    \includegraphics[width=0.5\textwidth]{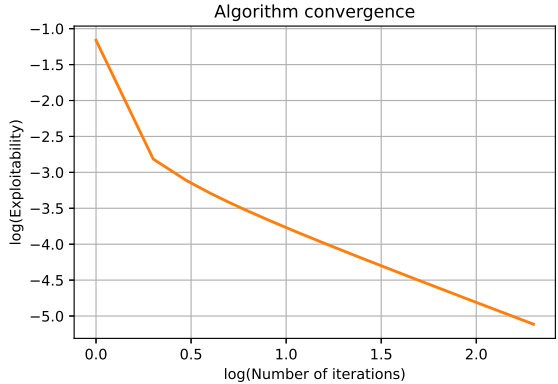}
    \caption{Convergence of the algorithm: log-log plot of the exploitability.}
    \label{error_os}
\end{figure}

\subsection{Compactness of $\mathcal{R}$ under the convergence in measure topology}\label{sec topo measure}

\noindent In this subsection, we prepare the ground for showing the convergence of our LPFP algorithm, by providing several results which are needed in the proof of the main result. In particular, we  show the compactness of the set $\mathcal{R}$ under the topology of the convergence in measure. This result also allows to prove the existence of an equilibrium under Assumption \ref{assump existence OS} only (see Theorem \ref{unbounded existence}), improving earlier results in \cite{bdt2020} and \cite{dlt2021}, which did not take into account a general dependence of the map $f$ on $m$ (resp. $g$ on $\mu$) nor coefficients with polynomial growth.\vspace{5pt}

\noindent Recall that by Theorem C.6 in \cite{dlt2021}, for $(\mu, m)\in \mathcal{R}$, there exists $(\Omega, \mathcal{F}, \mathbb F, \mathbb P, W, \tau, X)$, such that $(\Omega, \mathcal{F}, \mathbb F, \mathbb P)$ is a filtered probability space, $W$ is an $\mathbb F$-Brownian motion, $\tau$ is an $\mathbb F$-stopping time such that $\tau\leq T\wedge\tau_\mathcal{O}^{X}$ $\mathbb P$-a.s.~and $X$ is an $\mathbb F$-adapted process verifying
$$X_{t}= X_0 + \int_0^{t}b(s, X_s)ds + \int_0^{t} \sigma(s, X_s)dW_s, \quad \mathbb P \circ X_0^{-1}= m_0^*,$$
$$\mu =\mathbb P \circ (\tau, X_\tau)^{-1},\quad \text{and}\quad m_t(B) = m^\star_t(B),  \quad B\in \mathcal{B}(\bar{\mathcal{O}}), \quad t-a.e,$$
where 
$$m^{\star}_t(B):= \mathbb E^{\mathbb P}\left[ \mathds{1}_B(X_t)\mathds{1}_{t< \tau}\right].
$$
Seen as a mapping from $[0, T]$ to $\mathcal{P}^{sub}(\bar{\mathcal{O}})$ endowed with the topology of weak convergence, $m^\star$ is càdlàg. In fact, for any $\varphi\in C_b(\bar{\mathcal O})$ and $t_n\downarrow t$, by dominated convergence,
$$\lim_{n\rightarrow\infty}\int_{\bar{\mathcal{O}}}\varphi(x)m^{\star}_{t_n}(dx)=\lim_{n\rightarrow\infty}\mathbb E^{\mathbb P}\left[ \varphi(X_{t_n})\mathds{1}_{t_n< \tau}\right]=\mathbb E^{\mathbb P}\left[ \varphi(X_{t})\mathds{1}_{t< \tau}\right]=\int_{\bar{\mathcal{O}}}\varphi(x)m^{\star}_{t}(dx).$$
Similarly, for $t_n\uparrow t$,
$$\lim_{n\rightarrow\infty}\int_{\bar{\mathcal{O}}}\varphi(x)m^{\star}_{t_n}(dx)= \mathbb E^{\mathbb P}\left[ \varphi(X_{t})\mathds{1}_{t\leq \tau}\right].$$
In general we can not expect the measure $m^\star$ to be continuous (consider for example the measures associated to deterministic stopping times). The MFG interpretation of this time irregularity comes from the possible simultaneous exit of a significant amount of players. Without loss of generality, we can and will always consider the càdlàg representative of $m$, still denoted by $m$.\\

The following estimates will be useful to establish the main result of the subsection, Theorem \ref{R compact OS}.

\begin{lemma}[Estimates]\label{estimates R OS}
There exist constants $C_1$ and $C_2$ such that for all $(\mu, m)\in \mathcal{R}$ we have the following estimates:
\begin{enumerate}[(1)]
\item For all $t\in [0, T]$, $\int_{[0, T]\times \bar{\mathcal{O}}}|x|^q\mu(dt, dx) \leq C_1, \quad \int_{\bar{\mathcal{O}}}|x|^qm_t(dx)\leq C_1$.
\item For all $h\in [0, T]$, we have $\int_0^{T-h}d_{\text{BL}}(m_{t+h}, m_t)dt \leq C_2\sqrt{h}$, where $d_{\text{BL}}$ is the bounded Lipschitz distance (see Appendix \ref{app tau p}).
\end{enumerate}
\end{lemma}

\begin{proof}
Let $(\mu, m)\in \mathcal{R}$ and consider their probabilistic representation. Using standard estimates, we obtain the existence of a constant $C_1\geq 0$ such that $\mathbb E^\mathbb P\|X\|^q_{T} \leq C_1$. We deduce that 
$$\int_{[0, T]\times \bar{\mathcal{O}}}|x|^q\mu(dt, dx)=\mathbb E^{\mathbb P}[|X_{\tau}|^q] \leq \mathbb E^{\mathbb P}\left[\|X\|_T^q\right]\leq C_1,$$
$$\int_{\bar{\mathcal{O}}}|x|^qm_t(dx)=\mathbb E^{\mathbb P}[|X_t|^q\mathds{1}_{t<\tau}]\leq E^{\mathbb P}\left[\|X\|_T^q\right]\leq C_1, \quad t\in [0, T].$$
Let us show the second estimate. For $\phi\in \text{BL}(\bar{\mathcal{O}})$ such that $\|\phi\|_{\text{BL}}\leq 1$ (where $\|\|_{\text{BL}}$ denotes the bounded Lipschitz norm, see Appendix \ref{app tau p}), $h\in [0, T]$ and $t\in [0, T-h]$, we have: 
\begin{align*}
\int_{\bar{\mathcal{O}}} \phi(x)(m_{t+h}-m_t)(dx)&= \mathbb E^{\mathbb P}[\phi(X_{t+h})\mathds{1}_{t+h<\tau} - \phi(X_t)\mathds{1}_{t<\tau}]\\
&=\mathbb E^{\mathbb P}[(\phi(X_{t+h}) - \phi(X_{t})) \mathds{1}_{t+h<\tau}]  + \mathbb E^{\mathbb P}[\phi(X_{t})(\mathds{1}_{t+h<\tau} - \mathds{1}_{t<\tau})]\\
&\leq \mathbb E^{\mathbb P}[|\phi(X_{t+h}) - \phi(X_{t})|]  + \mathbb E^{\mathbb P}[|\phi(X_{t})|\mathds{1}_{t<\tau\leq t+h}]\\
&\leq \mathbb E^{\mathbb P}[|X_{t+h} - X_{t}|] + \mathbb E^{\mathbb P}[\mathds{1}_{t<\tau\leq t+h}].
\end{align*}
Taking the supremum over $\phi$ we get $d_{\text{BL}}(m_{t+h}, m_t)\leq \mathbb E^{\mathbb P}[|X_{t+h} - X_{t}|] + \mathbb E^{\mathbb P}[\mathds{1}_{t<\tau\leq t+h}]$.
We also have
\begin{align*}
\int_0^{T-h}\mathbb E^{\mathbb P}[\mathds{1}_{t<\tau\leq t+h}]dt \leq  \mathbb E^{\mathbb P}\left[ \tau\wedge (T-h)-0\vee (\tau-h)\right]\leq h.
\end{align*}
On the other hand, applying Jensen's inequality and Burkholder-Davis-Gundy inequality, we get $\mathbb E^{\mathbb P}\left|X_{t+h} - X_{t}\right|^2\leq C h$, for some constant $C\geq 0$. We deduce that there exists a constant $C_2$ such that $\int_0^{T-h}d_{\text{BL}}(m_{t+h}, m_t)dt \leq C_2\sqrt{h}$.
\end{proof}

Using the above estimates, one can show that for $(\mu, m)\in \mathcal{R}$ the flow of measures $m$ is càdlàg as a map from $[0, T]$ to $\mathcal{P}_p^{sub}(\bar{\mathcal{O}})$ endowed with the $\tau_p$-topology (weak convergence with $p$-growth). In particular $m$ is Borel measurable from $[0, T]$ to $\mathcal{P}_p^{sub}(\bar{\mathcal{O}})$. \vspace{5pt}

Now we study the topology of convergence in measure which is convenient since we will be able to extract subsequences converging in $\mathcal{P}_p^{sub}(\bar{\mathcal{O}})$ a.e. on $[0, T]$. This topology allows to have a general mean-field dependence on the reward function $f$, which was not the case in \cite{dlt2021}. Moreover, it permits to consider test functions which are only measurable in time, and therefore to use the topology of the stable convergence. More precisely, we say that $(m^n)_{n\geq 1}\subset V_p$ converges in $\bar \tau_p$ to $m\in V_p$ if for all test functions $\phi:[0, T]\times \bar{\mathcal{O}}\rightarrow \mathbb R$ jointly measurable, continuous in $x$ for each $t$ and with $p$-polynomial growth, we have
$$\lim_{n\rightarrow\infty} \int_0^T\int_{\bar{\mathcal{O}}}\phi(t, x)m_t^n(dx)dt = \int_0^T\int_{\bar{\mathcal{O}}}\phi(t, x)m_t(dx)dt.$$ 
We state this result in the next lemma, whose proof follows by Proposition \ref{p stable} and an intermediary application of Corollary 2.9 in \cite{jacod1981}.

\begin{lemma}[\textit{Stable convergence with polynomial growth}]\label{lemma stable OS}
On the set $\mathcal{R}$ we have the inclusion $\tau_p\otimes \bar \tau_p\subset \tau_p\otimes \tilde \tau_p$. In other words, the convergence in measure of the flow of measures implies the convergence in the stable topology.
\end{lemma}

For the sake of clarity, we give the definition of a coercive normal integrand, which can be found in \cite{rossi2003}.

\begin{definition}[Coercive normal integrands]\label{coer nor int}
$H:]0, T[\times E\rightarrow [0, \infty]$ is a coercive normal integrand if 
\begin{enumerate}[(1)]
\item It is measurable with respect to $\mathcal{L}\otimes \mathcal{B}(E)$, where $\mathcal{L}$ denotes the Lebesgue-measurable subsets of $]0, T[$.
\item The maps $x\mapsto H_t(x):=H(t, x)$ are lower semicontinuous for a.e. $t\in ]0, T[$.
\item The sets $\{x\in E: H_t(x)\leq c\}$ are compact for any $c\geq 0$ and for a.e. $t\in ]0, T[$.
\end{enumerate}
\end{definition}

\begin{lemma}\label{lemma coercive OS}
The map $H:\mathcal{P}^{sub}(\bar{\mathcal{O}})\rightarrow [0, \infty]$ defined by $H(m)=\int_{\bar{\mathcal{O}}}|x|^qm(dx)$, $m\in \mathcal{P}^{sub}(\bar{\mathcal{O}})$,
is a coercive normal integrand, where $\mathcal{P}^{sub}(\bar{\mathcal{O}})$ is endowed with the topology of weak convergence.
\end{lemma}

\begin{proof}
To show that $H$ is a coercive normal integrand it suffices to prove that $H$ is measurable and has compact level sets. Define the maps $H_k(m)=\int_{\bar{\mathcal{O}}}[|x|^q\wedge k] m(dx)$, $k\geq 1$. These functions are continuous for the topology of weak convergence. Moreover by the monotone convergence theorem $H_k(m)$ converges to $H(m)$ for each $m\in \mathcal{P}^{sub}(\bar{\mathcal{O}})$ as $k\rightarrow\infty$, which allows to conclude the measurability of $H$. Let us show that $H$ has compact level sets. Let $c\geq 0$ and $L_c:=\{m\in \mathcal{P}^{sub}(\bar{\mathcal{O}}):H(m)\leq c\}$. By definition 
$$\sup_{m\in L_c} \int_{\bar{\mathcal{O}}}|x|^qm(dx)\leq c,$$
which shows that $L_c$ is relatively compact in $\mathcal{P}^{sub}(\bar{\mathcal{O}})$. It remains to show that $L_c$ is closed. Let $(m^n)_{n\geq 1}\subset L_c$ converging to some $m\in \mathcal{P}^{sub}(\bar{\mathcal{O}})$. We have for each $n\geq 1$ and $k\geq 1$, $H_k(m^n)\leq H(m^n)\leq c$, and taking the limit $n\rightarrow\infty$, we get $H_k(m)\leq c$.
By the monotone convergence theorem we deduce that $m\in L_c$.
\end{proof}

\begin{theorem}[\textit{Compactness of $\mathcal{R}$ in $\tau_p\otimes \tilde\tau_p$}]\label{R compact OS}
The topological space $(\mathcal{R}, \tau_p\otimes \tilde \tau_p)$ is compact.
\end{theorem}

\begin{proof}
Since the space $(\mathcal{R}, \tau_p\otimes \tilde \tau_p)$ is metrizable, it suffices to show that it is sequentially compact. Consider a sequence $(\mu^n, m^n)_{n\geq 1}\subset \mathcal{R}$. Using the estimate (1) from Lemma \ref{estimates R OS}
$$\int_{[0, T]\times \bar{\mathcal{O}}}|x|^q\mu^n(dt, dx) \leq C_1,$$
we get by Corollary \ref{compact M_p} that up to a subsequence, $(\mu^n)_{n\geq 1}$ converges to some $\mu\in \mathcal{P}_p([0, T]\times \bar{\mathcal{O}})$ in $\tau_p$. Let us show now that we can extract a further subsequence such that $(m^n)_{n\geq 1}$ converges to some $m$ in $\tilde \tau_p$. To prove this, we use the relative compactness criterion given in Theorem 2 and Extension 1 in \cite{rossi2003} for the convergence in measure topology. Let $H$ be the map defined in Lemma \ref{lemma coercive OS}. Using the first estimates of Lemma \ref{estimates R OS},
$$\sup_{n\geq 1}\int_0^TH(m^{n}_t)dt=\sup_{n\geq 1}\int_0^T\int_{\bar{\mathcal{O}}}|x|^qm^{n}_t(dx)dt \leq C_1T.$$
Now, using the second estimate of Lemma \ref{estimates R OS}, $\lim_{h\downarrow 0} \sup_{n\geq 1} \int_0^{T-h} d_{\text{BL}}(m^{n}_{t+h}, m^{n}_t)dt=0$. By Theorem 2 and Extension 1 in \cite{rossi2003}, up to a subsequence, $(m^{n})_{n\geq 1}$ converges to some $m \in M([0, T];\mathcal{P}^{sub}(\bar{\mathcal{O}}))$ in measure. Up to another subsequence, $(m_t^{n})_{n\geq 1}$ converges weakly to $m_t$ $t$-a.e. on $[0, T]$. Since for each $t\in [0, T]$,
$$\sup_{n\geq 1}\int_{\bar{\mathcal{O}}}|x|^q m^{n}_t(dx)\leq C_1,$$
$(m_t^{n})_{n\geq 1}$ converges to $m_t$ in $\tau_p$ $t$-a.e. on $[0, T]$, and in particular $(m^{n})_{n\geq 1}$ converges to $m$ in $\tilde \tau_p$. Finally, by Lemma \ref{lemma stable OS} we get $m^n\rightarrow m$ in $\bar\tau_p$ and we can pass easily to the limit in the constraint to conclude that $(\mu, m)\in \mathcal{R}$. 
\end{proof}

Recall that if a set is compact and Hausdorff under two comparable topologies, then both topologies coincide (see \cite{munkres2000} Chapter 3, Exercise 1.(b) p.168). As a consequence of the above theorem, we get the following result which will be useful to consider different metrics on the space $\mathcal{R}$ in order to show the convergence of the algorithm.

\begin{corollary}\label{rho distance}
On the set $\mathcal{R}$ the topologies $\tau_0\otimes \tau_0$, $\tau_p\otimes \tau_p$, $\tau_p\otimes \bar \tau_p$ and $\tau_p\otimes \tilde\tau_p$ coincide.
\end{corollary}

The existence of a maximizer of the best response map $\Theta:\mathcal{R}\rightarrow 2^\mathcal{R}$ defined by
$$\Theta(\bar \mu, \bar m)=\argmax_{(\mu, m)\in \mathcal{R}}\Gamma[\bar \mu, \bar m](\mu, m), \quad (\bar \mu, \bar m)\in \mathcal{R}$$ follows by the same arguments as in Theorem 2.14 from \cite{dlt2021}. Then, by applying the Kakutani-Fan-Glicksberg’s fixed point theorem for set-valued maps (see Theorem 3.11 in \cite{dlt2021}) together with an intermediary application of Lemmas \ref{slutsky_p} and \ref{slutsky_stable_p}, we get the existence of an LP MFG equilibrium.

\begin{theorem}[\textit{Existence of LP MFG equilibria}]\label{unbounded existence}
Under Assumption \ref{assump existence OS}, there exists an LP MFG Nash equilibrium.
\end{theorem}

To show the convergence of the algorithm, we use a precise metric on the set $\mathcal{R}$, denoted by $\rho$, which is introduced below. In particular,this metric is used in Corollary \ref{unif continuous OS}, where we show the convergence to zero of the distance $\rho$ between two successive best responses. We denote by $W_1$  the $1$-Wasserstein metric on $\mathcal{P}_1([0, T]\times \bar{\mathcal{O}})$. We also make use of an analogue of the $1$-Wasserstein metric on $\mathcal{P}_1^{sub}(\bar{\mathcal{O}})$, which is denoted by $W_1'$ and constructed as in Appendix B of \cite{claisse2019}. This metric depends on some arbitrary reference point $x_0\in \bar{\mathcal{O}}$ (which is fixed for the rest of the paper) and metrizes the topology $\tau_1$ in $\mathcal{P}_1^{sub}(\bar{\mathcal{O}})$ described in Appendix \ref{app tau p} (see Lemma B.2. in \cite{claisse2019}). In particular, we use the following Kantorovich duality type result (Lemma B.1 in \cite{claisse2019}):
$$W_1'(m, m')=\sup_{\phi\in \text{Lip}_1(\bar{\mathcal{O}}, x_0)}\int_{\bar{\mathcal{O}}}\phi(x)(m-m')(dx) + |m(\bar{\mathcal{O}})- m'(\bar{\mathcal{O}})|,\quad m, m'\in \mathcal{P}_1^{sub}(\bar{\mathcal{O}}),$$
where $\text{Lip}_1(\bar{\mathcal{O}}, x_0)$ is the set of all functions $\phi:\bar{\mathcal{O}}\rightarrow \mathbb R$ with Lipschitz constant smaller or equal to 1 and such that $\phi(x_0)=0$.

\begin{proposition}[\textit{Metric $\rho$}]\label{distance tau_1}
Any of the topologies on the set $\mathcal{R}$ considered in Corollary \ref{rho distance} is induced by the metric $\rho$, given by $ \rho((\mu, m), (\mu', m'))=W_1(\mu, \mu') + d_M(m, m')$, with $(\mu, m), (\mu', m')\in \mathcal{R}$, and 
\begin{equation}\label{distance measure}
d_M(m, m'):=\int_0^T W_1'(m_t, m_t')dt.    
\end{equation}
\end{proposition}

\begin{proof}
Let us prove that the metric $\rho$ metrizes the topology $\tau_1 \otimes \tilde{\tau}_1$ on $\mathcal{R}$. To do so, by Lemma \ref{bound conv in meas}, it suffices to show that there exists a constant $C\geq 0$ such that for each $(\mu, m)\in \mathcal{R}$ we have $t$-a.e. on $[0, T]$, $W_1'(m_t, \mathbf{0})\leq C$, where $\mathbf{0}$ denotes the null measure on $\bar{\mathcal{O}}$. Let $(\mu, m)\in \mathcal{R}$. By Lemma \ref{estimates R OS}, there exists a constant $C\geq 0$ such that for all $t\in [0, T]$ we have
$$\int_{\bar{\mathcal{O}}}|x|m_t(dx)\leq C.$$
We obtain for all $t\in [0, T]$,
\begin{align*}
W_1'(m_t, \mathbf{0})&= \sup_{\phi\in \text{Lip}_1(\bar{\mathcal{O}}, x_0)}\int_{\bar{\mathcal{O}}}\phi(x)m_t(dx) + m_t(\bar{\mathcal{O}})\leq \int_{\bar{\mathcal{O}}}|x|m_t(dx) + |x_0| + 1\leq C + |x_0| + 1.
\end{align*}
\end{proof}

\subsection{Convergence of the algorithm}

In this subsection,  Assumption \ref{assump existence OS} and Assumption \ref{main assump OS} are in force. We recall the quantities computed by the algorithm: for $N\geq 0$,
$$(\mu^{(N+1)}, m^{(N+1)}):=\Theta(\bar \mu^{(N)}, \bar m^{(N)}),$$
$$(\bar \mu^{(N+1)}, \bar m^{(N+1)}):=\frac{N}{N+1}(\bar \mu^{(N)}, \bar m^{(N)}) + \frac{1}{N+1}(\mu^{(N+1)},  m^{(N+1)})=\frac{1}{N + 1}\sum_{k=1}^{N+1}(\mu^{(k)},  m^{(k)}).$$
Since all these tuples are in $\mathcal{R}$, the measures $(m^{(N)})_N$ and $(\bar{m}^{(N)})_N$ admit càdlàg representatives. Without loss of generality, we  consider the càdlàg representatives, for which the same notation is used.\\

\noindent We first establish some useful estimates, as well as the continuity of the best response map $\Theta$.

\subsubsection{Estimates and regularity of the best response map}

Using the definitions of $(\bar{\mu}^{(N)}, \bar{m}^{(N)})$ given by the algorithm procedure, we derive the following estimates between two successive output measures.

\begin{lemma}[\textit{Estimates between two successive output measures $(\bar \mu^{(N)}, \bar m^{(N)})$}]\label{R estimates OS}
For all $N\geq 1$, we have the following estimates:
\begin{enumerate}[(i)]
\item There exists a constant $C_3\geq 0$ (independent of $N$) such that
$$W_1(\bar \mu^{(N)}, \bar \mu^{(N+1)})\leq \frac{C_3}{N},\quad d_{M}(\bar m^{(N)}, \bar m^{(N+1)})\leq \frac{C_3}{N}.$$

\item There exists a constant $C_4\geq 0$ (independent of $N$) such that for all $t\in [0, T]$,
$$\int_{[0, T]\times \bar{\mathcal{O}}}(1+|x|^p)|\bar \mu^{(N+1)}-\bar\mu^{(N)}|(dt, dx)\leq \frac{C_4}{N},\quad \int_{\bar{\mathcal{O}}}(1+|x|^p)|\bar m^{(N+1)}_t-\bar m^{(N)}_t|(dx)\leq \frac{C_4}{N}.$$
\end{enumerate}
\end{lemma}

\begin{proof}

\begin{enumerate}[(i)]
\item Note that $\bar \mu^{(N+1)} - \bar \mu^{(N)} = \frac{1}{N+1}\left[\mu^{(N+1)} -\bar \mu^{(N)}\right]$. Let $\varphi\in \text{Lip}_1([0, T]\times \bar{\mathcal{O}})$, i.e. a $1$-Lipschitz function. Using Lemma \ref{estimates R OS} (i),
\begin{align*}
\langle \varphi,\bar\mu^{(N+1)}-\bar\mu^{(N)}\rangle &= \frac{1}{N+1} \langle \varphi, \mu^{(N+1)}-\bar\mu^{(N)}\rangle = \frac{1}{N+1} \langle \varphi - \varphi(0, x_0), \mu^{(N+1)}-\bar\mu^{(N)}\rangle\\
&\leq \frac{1}{N+1}\langle |\varphi - \varphi(0, x_0)|, \mu^{(N+1)}+\bar\mu^{(N)}\rangle\\
&\leq \frac{C}{N},
\end{align*}
for some constant $C\geq 0$ independent from $N$ and $\varphi$. Taking the supremum over $\varphi$ we obtain the result. 

Analogously, $\bar m^{(N+1)} - \bar m^{(N)} = \frac{1}{N+1}\left[m^{(N+1)} -\bar m^{(N)}\right]$. Let $\varphi\in \text{Lip}_1(\bar{\mathcal{O}}, x_0)$, i.e. a $1$-Lipschitz function with $\varphi(x_0)=0$. Using again Lemma \ref{estimates R OS} (i), for all $t\in [0, T]$,
\begin{align*}
\langle \varphi,\bar m^{(N+1)}_t-\bar m^{(N)}_t\rangle &= \frac{1}{N+1} \langle \varphi, m^{(N+1)}_t-\bar m^{(N)}_t\rangle \\
&\leq \frac{1}{N+1}\langle |\varphi|, m^{(N+1)}_t+\bar m^{(N)}_t\rangle\\
&\leq \frac{C}{N},
\end{align*}
for some constant $C\geq 0$ independent from $N$, $\varphi$ and $t$. Taking the supremum over $\varphi$ and integrating over $t$ we get the claimed estimate.

\item By Lemma \ref{estimates R OS}, we get for all $t\in [0, T]$:
\begin{align*}
\int_{\bar{\mathcal{O}}}(1+|x|^p)|\bar m^{(N+1)}_t-\bar m^{(N)}_t|(dx)\leq \frac{1}{N+1}\int_{\bar{\mathcal{O}}}(1+|x|^p)(m^{(N+1)}_t+\bar m^{(N)}_t)(dx)\leq \frac{C}{N}. 
\end{align*}
Similarly, by Lemma \ref{estimates R OS}, we get
$\int_{[0, T]\times \bar{\mathcal{O}}}(1+|x|^p)|\bar \mu^{(N+1)}-\bar\mu^{(N)}|(dt, dx) \leq\frac{C}{N}$.
\end{enumerate}
\end{proof}

\noindent We establish below the following estimates on the reward map.

\begin{lemma}[\textit{Estimates on the reward map}]\label{lemma lipschitz OS}
There exist constants $C_f$ and $C_g$ such that for all $N\geq 1$
$$
\langle f(\bar m^{(N+1)}) - f(\bar m^{(N)}),m^{(N+2)}-m^{(N+1)}\rangle 
\leq \frac{C_f}{N}d_M(m^{(N+1)}, m^{(N+2)}),
$$
$$
\langle g(\bar \mu^{(N+1)}) - g(\bar \mu^{(N)}),\mu^{(N+2)}-\mu^{(N+1)}\rangle \leq \frac{C_g}{N}W_1(\mu^{(N+1)}, \mu^{(N+2)}).
$$
\end{lemma}

\begin{proof}
Let us first show that the function $\varphi_N:[0, T]\times \bar{\mathcal{O}}\rightarrow \mathbb R$ defined by 
$$\varphi_N(t, x):=f(t, x, \bar m^{(N+1)}_t)-f(t, x, \bar m^{(N)}_t)$$
is a $C/N$-Lipschitz continuous function in $x$ uniformly on $t$, for some constant $C \geq 0$. Indeed, by item (3) in Assumption \ref{main assump OS} and Lemma \ref{R estimates OS}, for each $t\in [0, T]$ and $x, x'\in \bar{\mathcal{O}}$, we have
\begin{align*}
|\varphi_N(t, x)-\varphi_N(t, x')|&=|f(t, x, \bar m^{(N+1)}_t)-f(t, x, \bar m^{(N)}_t) - f(t, x', \bar m^{(N+1)}_t) + f(t, x', \bar m^{(N)}_t)| \\
& \leq c_f|x-x'|\int_{\bar{\mathcal{O}}}(1+|z|^p) |\bar m^{(N+1)}_t-  \bar m^{(N)}_t|(dz) \leq \frac{C}{N}|x-x'|.
\end{align*}
The same holds for the function $(t, x)\mapsto \varphi_N(t, x) - \varphi_N(t, x_0)$, which is equal to $0$ at $x_0$. By definition of $W_1'$,
$$W_1'(m^{(N+2)}_t, m^{(N+1)}_t)=\sup_{\phi\in \text{Lip}_1(\bar{\mathcal{O}}, x_0)}\int_{\bar{\mathcal{O}}}\phi(x)(m^{(N+2)}_t-m^{(N+1)}_t)(dx) + |m^{(N+2)}_t(\bar{\mathcal{O}})- m^{(N+1)}_t(\bar{\mathcal{O}})|.$$
Furthermore, by item (3) from Assumption \ref{main assump OS} and Lemma \ref{R estimates OS}, we get that for $t\in [0, T]$,
$$|\varphi_N(t, x_0)|=|f(t, x_0, \bar m^{(N+1)}_t)-f(t, x_0, \bar m^{(N)}_t)|\leq c_f(1+|x_0|) \int_{\bar{\mathcal{O}}}(1+|z|^p) |\bar m^{(N+1)}_t-  \bar m^{(N)}_t|(dz)\leq \frac{C'}{N}.$$
We derive 
\begin{align*}
&\langle f(\bar m^{(N+1)}) - f(\bar m^{(N)}),m^{(N+2)}-m^{(N+1)}\rangle\\
& \quad =  \int_0^T\int_{\bar{\mathcal{O}}}(\varphi_N(t, x)-\varphi_N(t, x_0))(m^{(N+2)}_t-m^{(N+1)}_t)(dx)dt \\
&\quad \quad + \int_0^T\int_{\bar{\mathcal{O}}}\varphi_N(t, x_0)(m^{(N+2)}_t-m^{(N+1)}_t)(dx)dt \\
&\quad \leq \frac{C}{N} \int_0^TW_1'(m^{(N+2)}_t, m^{(N+1)}_t)dt + \frac{C'}{N}\int_0^TW_1'(m^{(N+2)}_t, m^{(N+1)}_t)dt\\
&\quad \leq \frac{C +C'}{N}d_M(m^{(N+2)}, m^{(N+1)}),
\end{align*}
where the last inequality follows by definition of the metric $d_M$ given by \eqref{distance measure}.

\noindent Let us show the second estimate. To this purpose, we consider the function $\psi_N:[0, T]\times \bar{\mathcal{O}}\rightarrow\mathbb R$ defined by 
$$\psi_N(t, x):=g(t, x, \bar \mu^{(N+1)}) - g(t, x, \bar \mu^{(N)})$$ and show that it 
is a $C''/N$-Lipschitz continuous function, for some constant $C''\geq 0$. Indeed, 
by item (3) in Assumption \ref{main assump OS} and Lemma \ref{R estimates OS}, for each $t, t'\in [0, T]$, $x, x'\in \bar{\mathcal{O}}$, we have
\begin{align*}
|\psi_N(t, x)-\psi_N(t', x')|&=|g(t, x, \bar \mu^{(N+1)}) - g(t, x, \bar \mu^{(N)}) - g(t', x', \bar \mu^{(N+1)}) + g(t', x', \bar \mu^{(N)})|\\
&\quad \leq c_g(|t-t'|+|x-x'|)\int_{[0, T]\times \bar{\mathcal{O}}}(1+|z|^p)|\bar \mu^{(N+1)}-\bar \mu^{(N)}|(ds, dz)\\
&\quad \leq \frac{C''}{N}(|t-t'|+|x-x'|). 
\end{align*}
By Kantorovich's duality theorem,
\begin{align*}
&\langle g(\bar \mu^{(N+1)}) - g(\bar \mu^{(N)}),\mu^{(N+2)}-\mu^{(N+1)}\rangle = \langle \psi_N ,\mu^{(N+2)}-\mu^{(N+1)}\rangle \leq \frac{C''}{N}W_1 (\mu^{(N+2)}, \mu^{(N+1)}).
\end{align*}
\end{proof}

In the following lemma, we prove the continuity of the best response map.

\begin{lemma}[\textit{Continuity of the best response map}]\label{Theta cont OS}
The function $\Theta$ is continuous on $\mathcal{R}$ with respect to all topologies listed in Corollary \ref{rho distance}.
\end{lemma}

\begin{proof}
Let $(\bar \mu^n, \bar m^n)_{n\geq 1}\subset \mathcal{R}$ be a sequence converging to $(\bar \mu, \bar m)\in \mathcal{R}$. Define $(\mu^n, m^n):=\Theta(\bar \mu^n, \bar m^n)\in\mathcal{R}$ and let $(\mu, m)\in \mathcal{R}$ be a cluster point of the sequence $(\mu^n, m^n)_{n\geq 1}$ (which exists since $\mathcal{R}$ is compact). Up to taking a subsequence, we assume that the entire sequence converges to $(\mu, m)$. Let $(\tilde\mu, \tilde m)\in \mathcal{R}$, we have to show that $\Gamma[\bar \mu, \bar m](\tilde\mu, \tilde m)\leq \Gamma[\bar \mu, \bar m](\mu, m)$. By definition of $\Theta$, $\Gamma[\bar \mu^n, \bar m^n](\tilde\mu, \tilde m)\leq \Gamma[\bar \mu^n, \bar m^n](\mu^n, m^n)$. Taking the limit in the above inequality as $n\rightarrow\infty$ (by an intermediary application of Lemmas \ref{slutsky_p} and \ref{slutsky_stable_p}), we obtain $\Gamma[\bar \mu, \bar m](\tilde\mu, \tilde m)\leq \Gamma[\bar \mu, \bar m](\mu, m)$, which, by uniqueness of the best response, shows that $(\mu, m)=\Theta(\bar \mu, \bar m)$.
\end{proof}

Using Lemma \ref{R estimates OS} and Lemma \ref{Theta cont OS}, we derive the following result.

\begin{corollary}[\textit{Proximity between two successive best responses}]\label{unif continuous OS}
We have
$$\lim_{N\rightarrow \infty} W_1(\mu^{(N)}, \mu^{(N+1)})=0 \quad \text{and} \quad \lim_{N\rightarrow \infty} d_M(m^{(N)}, m^{(N+1)}) = 0.$$
\end{corollary}
\begin{proof}
Recall the metric $\rho$ on $\mathcal{R}$ defined in Proposition \ref{distance tau_1}. Viewing $\Theta$ as a function between the metric spaces $(\mathcal{R}, \rho)$ and $(\mathcal{R}, \rho)$, it is uniformly continuous since it is continuous by Lemma \ref{Theta cont OS} and $(\mathcal{R}, \rho)$ is a compact metric space by Theorem \ref{R compact OS}. By Lemma \ref{R estimates OS} (i),
$$\rho((\bar \mu^{(N)}, \bar m^{(N)}), (\bar \mu^{(N+1)}, \bar m^{(N+1)}))\underset{N\rightarrow\infty}{\longrightarrow}0.$$
We get by the sequential characterization of the uniform continuity that
$$\rho((\mu^{(N+1)}, m^{(N+1)}), (\mu^{(N+2)}, m^{(N+2)})) =\rho(\Theta(\bar \mu^{(N)}, \bar m^{(N)}), \Theta(\bar \mu^{(N+1)}, \bar m^{(N+1)}))\underset{N\rightarrow\infty}{\longrightarrow}0.$$
\end{proof}

\subsubsection{Main convergence result}\label{sec main theo}

In this section, we prove the convergence of the algorithm. To do so, we first introduce the following sequence of real numbers $\varepsilon_N$, which quantifies how far $(\bar \mu^{(N)}, \bar m^{(N)})$ is from being the best response when the reward maps depend on $(\bar \mu^{(N)}, \bar m^{(N)})$. Therefore, $\varepsilon_N$ quantifies the proximity of $(\bar \mu^{(N)}, \bar m^{(N)})$  from an \textit{LP MFG Nash equilibrium}.

\begin{definition}[\textit{Exploitability}]
We define the sequence of real numbers $(\varepsilon_N)_{N\geq 1}$ by
\begin{equation}\label{phi n OS}
\varepsilon_N=\langle f(\bar m^{(N)}),m^{(N+1)}-\bar m^{(N)}\rangle 
+ \langle g(\bar \mu^{(N)}),\mu^{(N+1)}-\bar \mu^{(N)}\rangle\geq 0.    
\end{equation}
In particular, $(\bar \mu^{(N)}, \bar m^{(N)})$ is an \textit{$\varepsilon_N$-LP MFG Nash equilibrium} and we will show in the next theorem that $\varepsilon_N\rightarrow 0$ as $N\rightarrow \infty$.
\end{definition}

\begin{proof}[Proof of Theorem \ref{main theo OS}]
In the proof we will denote by $C\geq 0$ a generic constant which may change from line to line. Recall the expression of $\varepsilon_N$ from \eqref{phi n OS}. We can rewrite $\varepsilon_N$ as
$$\varepsilon_N=\Gamma[\bar \mu^{(N)}, \bar m^{(N)}](\mu^{(N+1)}, m^{(N+1)}) - \Gamma[\bar \mu^{(N)}, \bar m^{(N)}](\bar \mu^{(N)}, \bar m^{(N)}).$$
Now we have
\begin{multline*}
\varepsilon_{N+1}-\varepsilon_N= \Gamma[\bar \mu^{(N+1)}, \bar m^{(N+1)}](\mu^{(N+2)}, m^{(N+2)}) - \Gamma[\bar \mu^{(N+1)}, \bar m^{(N+1)}](\bar \mu^{(N+1)}, \bar m^{(N+1)}) \\- \Gamma[\bar \mu^{(N)}, \bar m^{(N)}](\mu^{(N+1)}, m^{(N+1)}) + \Gamma[\bar \mu^{(N)}, \bar m^{(N)}](\bar \mu^{(N)}, \bar m^{(N)}).
\end{multline*}
Define
$$\varepsilon_N^{(1)}:= \Gamma[\bar \mu^{(N)}, \bar m^{(N)}](\bar \mu^{(N)}, \bar m^{(N)}) - \Gamma[\bar \mu^{(N+1)}, \bar m^{(N+1)}](\bar \mu^{(N+1)}, \bar m^{(N+1)}),$$
$$\varepsilon_N^{(2)}:= \Gamma[\bar \mu^{(N+1)}, \bar m^{(N+1)}](\mu^{(N+2)}, m^{(N+2)}) - \Gamma[\bar \mu^{(N)}, \bar m^{(N)}](\mu^{(N+1)}, m^{(N+1)}).$$
Then $\varepsilon_{N+1}-\varepsilon_N = \varepsilon_N^{(1)} +\varepsilon_N^{(2)}$. Let us make some estimates of these two quantities.
Using Lemma \ref{R estimates OS} (ii), for each $t\in [0, T]$ and $x\in \bar{\mathcal{O}}$ we obtain
$$|f(t, x, \bar m_t^{(N+1)}) - f(t, x, \bar m_t^{(N)})|\leq c_f(1+|x|) \int_{\bar{\mathcal{O}}}(1+|z|^p)|\bar m_t^{(N+1)}-\bar m_t^{(N)}|(dz)\leq \frac{C}{N}(1+|x|),$$
and therefore we get by Lemma \ref{estimates R OS} (i)
\begin{align*}
-\frac{1}{N+1}\langle f(\bar m^{(N+1)})-f(\bar m^{(N)}), m^{(N+1)}-\bar m^{(N)} \rangle &\leq \frac{1}{N+1}\langle |f(\bar m^{(N+1)})-f(\bar m^{(N)})|, m^{(N+1)}+\bar m^{(N)} \rangle\\
& \leq \frac{C}{N^2}.
\end{align*}
We deduce that
\begin{align*}
&\langle f(\bar m^{(N)}), \bar m^{(N)} \rangle - \langle f(\bar m^{(N+1)}), \bar m^{(N+1)}\rangle = \langle f(\bar m^{(N)}), \bar m^{(N)} \rangle  \\
&\quad - \langle f(\bar m^{(N+1)}), \bar m^{(N)} + \frac{1}{N+1}(m^{(N+1)}-\bar m^{(N)}) \rangle \\
& = \langle f(\bar m^{(N)}) - f(\bar m^{(N+1)}), \bar m^{(N)} \rangle - \frac{1}{N+1}\langle f(\bar m^{(N+1)}), m^{(N+1)}-\bar m^{(N)} \rangle\\
&\leq \langle f(\bar m^{(N)}) - f(\bar m^{(N+1)}), \bar m^{(N)} \rangle - \frac{1}{N+1}\langle f(\bar m^{(N)}), m^{(N+1)}-\bar m^{(N)} \rangle + \frac{C}{N^2}.
\end{align*}
Analogously,
\begin{align*}
&\langle g(\bar \mu^{(N)}), \bar \mu^{(N)} \rangle - \langle g(\bar \mu^{(N+1)}), \bar \mu^{(N+1)}\rangle \\
&\leq \langle g(\bar \mu^{(N)}) - g(\bar \mu^{(N+1)}), \bar \mu^{(N)} \rangle - \frac{1}{N+1}\langle g(\bar \mu^{(N)}), \mu^{(N+1)}-\bar \mu^{(N)} \rangle + \frac{C}{N^2}.
\end{align*}
Therefore,
\begin{align*}
\varepsilon_N^{(1)}  
& = \langle f(\bar m^{(N)}), \bar m^{(N)} \rangle + \langle g(\bar \mu^{(N)}), \bar \mu^{(N)} \rangle - \langle f(\bar m^{(N+1)}), \bar m^{(N+1)} \rangle - \langle g(\bar \mu^{(N+1)}), \bar \mu^{(N+1)} \rangle\\
&\leq \langle f(\bar m^{(N)}) - f(\bar m^{(N+1)}), \bar m^{(N)} \rangle + \langle g(\bar \mu^{(N)}) - g(\bar \mu^{(N+1)}), \bar \mu^{(N)} \rangle -\frac{\varepsilon_N}{N+1} +\frac{C}{N^2}.
\end{align*}
On the other hand,
\begin{align*}
\varepsilon_N^{(2)} & =\Gamma[\bar \mu^{(N+1)}, \bar m^{(N+1)}](\mu^{(N+2)}, m^{(N+2)}) - \Gamma[\bar \mu^{(N)}, \bar m^{(N)}](\mu^{(N+1)}, m^{(N+1)})\\
&\leq \Gamma[\bar \mu^{(N+1)}, \bar m^{(N+1)}](\mu^{(N+2)}, m^{(N+2)}) - \Gamma[\bar \mu^{(N)}, \bar m^{(N)}](\mu^{(N+2)}, m^{(N+2)})\\
&=\langle f(\bar m^{(N+1)}) - f(\bar m^{(N)}), m^{(N+2)}\rangle + \langle g(\bar \mu^{(N+1)}) - g(\bar \mu^{(N)}), \mu^{(N+2)}\rangle \\
&= \langle f(\bar m^{(N+1)}) - f(\bar m^{(N)}), m^{(N+1)}\rangle + \langle g(\bar \mu^{(N+1)}) - g(\bar \mu^{(N)}), \mu^{(N+1)}\rangle \\
&\quad + \langle f(\bar m^{(N+1)}) - f(\bar m^{(N)}), m^{(N+2)}-m^{(N+1)}\rangle + \langle g(\bar \mu^{(N+1)}) - g(\bar \mu^{(N)}), \mu^{(N+2)} - \mu^{(N+1)}\rangle.
\end{align*}
Now, by Lemma \ref{lemma lipschitz OS},
$$
\langle f(\bar m^{(N+1)}) - f(\bar m^{(N)}), m^{(N+2)}-m^{(N+1)}\rangle \leq \frac{C_f}{N}d_M(m^{(N+1)}, m^{(N+2)}),
$$
$$\langle g(\bar \mu^{(N+1)}) - g(\bar \mu^{(N)}), \mu^{(N+2)} - \mu^{(N+1)}\rangle\leq \frac{C_g}{N}W_1(\mu^{(N+1)}, \mu^{(N+2)}).$$
Therefore
\begin{align*}
\varepsilon_N^{(2)} &\leq \langle f(\bar m^{(N+1)}) - f(\bar m^{(N)}), m^{(N+1)}\rangle + \langle g(\bar \mu^{(N+1)}) - g(\bar \mu^{(N)}), \mu^{(N+1)}\rangle\\
&\quad +\frac{C}{N}(d_M(m^{(N+1)}, m^{(N+2)}) + W_1(\mu^{(N+1)}, \mu^{(N+2)})).
\end{align*}
Letting 
$$\delta_N=C\left[ d_M(m^{(N+1)}, m^{(N+2)}) + W_1(\mu^{(N+1)}, \mu^{(N+2)}) +  \frac{1}{N}\right],$$
we get
\begin{align*}
&\varepsilon_{N+1} - \varepsilon_N = \varepsilon_N^{(1)} + \varepsilon_N^{(2)}\leq \langle f(\bar m^{(N)}) - f(\bar m^{(N+1)}), \bar m^{(N)} \rangle + \langle g(\bar \mu^{(N)}) - g(\bar \mu^{(N+1)}), \bar \mu^{(N)} \rangle -\frac{\varepsilon_N}{N+1}\\
&\quad + \langle f(\bar m^{(N+1)}) - f(\bar m^{(N)}), m^{(N+1)} \rangle + \langle g(\bar \mu^{(N+1)}) - g(\bar \mu^{(N)}), \mu^{(N+1)} \rangle + \frac{\delta_N}{N}\\
& = \langle f(\bar m^{(N+1)}) - f(\bar m^{(N)}), m^{(N+1)} - \bar m^{(N)}\rangle + \langle g(\bar \mu^{(N+1)}) - g(\bar \mu^{(N)}), \mu^{(N+1)} - \bar \mu^{(N)} \rangle -\frac{\varepsilon_N}{N+1} + \frac{\delta_N}{N}\\
& = (N+1)\left[ \langle f(\bar m^{(N+1)}) - f(\bar m^{(N)}), \bar m^{(N+1)} - \bar m^{(N)}\rangle + \langle g(\bar \mu^{(N+1)}) - g(\bar \mu^{(N)}), \bar \mu^{(N+1)} - \bar \mu^{(N)} \rangle \right] \\
&\quad -\frac{\varepsilon_N}{N+1} + \frac{\delta_N}{N}\\
&\leq -\frac{\varepsilon_N}{N+1} + \frac{\delta_N}{N},
\end{align*}
where the last inequality comes from the Lasry–Lions monotonicity condition. Observe that $\delta_N\rightarrow 0$ by Corollary \ref{unif continuous OS}. By Lemma 3.1 in \cite{hadikhanloo2017}, we conclude that $\varepsilon_N\rightarrow 0$ as $N\rightarrow \infty$. Let $((\mu, m), (\bar \mu, \bar m))$ be a cluster point of the sequence $((\mu^{(N+1)}, m^{(N+1)}), (\bar \mu^{(N)}, \bar m^{(N)}))_{N\geq 1}$ for the topology $\tau_p\otimes \tilde \tau_p$ and let us show that $(\mu, m)=(\bar \mu, \bar m)$, which implies that $(\mu, m)$ is an LP MFG Nash equilibrium by continuity of the map $\Theta$ (see Lemma \ref{Theta cont OS}). First note that since $(\mu^{(N+1)}, m^{(N+1)})=\Theta (\bar \mu^{(N)}, \bar m^{(N)})$, and $\Theta$ is continuous, we obtain $(\mu, m)=\Theta (\bar \mu, \bar m)$. Let $(\tilde \mu, \tilde m)\in \mathcal{R}$. We have
$$\Gamma[\bar \mu^{(N)}, \bar m^{(N)}](\mu^{(N+1)}, m^{(N+1)})\geq \Gamma[\bar \mu^{(N)}, \bar m^{(N)}](\tilde \mu, \tilde m).$$
By definition of $\varepsilon_N$,
$$\Gamma[\bar \mu^{(N)}, \bar m^{(N)}](\bar \mu^{(N)}, \bar m^{(N)})\geq \Gamma[\bar \mu^{(N)}, \bar m^{(N)}](\tilde \mu, \tilde m) -\varepsilon_N.$$
Taking the limit $N\rightarrow \infty$, we obtain $\Gamma[\bar \mu, \bar m](\bar \mu, \bar m)\geq \Gamma[\bar \mu, \bar m](\tilde \mu, \tilde m)$. Since $(\tilde\mu, \tilde m)$ was arbitrary in $\mathcal{R}$, we get $(\bar \mu, \bar m)=\Theta (\bar \mu, \bar m)=(\mu, m)$, i.e. $(\bar \mu, \bar m)=(\mu, m)$ is the unique LP MFG Nash equilibrium.
\end{proof}

\begin{remark}
The proof follows some of the steps given in \cite{hadikhanloo2017}, but is based on some new results due to our setting of optimal stopping MFGs (in particular, the flow of measures is discontinuous). More precisely, one needs to establish specific estimates using appropriate distances.
\end{remark}

\section{Linear programming algorithm for MFGs with \textit{pure control} and \textit{absorption}}

In this section, we illustrate the LPFP algorithm for MFGs with pure control and absorption, its convergence following by the same approach developed in the case of optimal stopping (see Theorem \ref{main theo OS}). In the setting of MFGs with pure control and absorption, the players control their dynamics up to the exit time from a given set $\mathcal{O}$, when they leave the game. For the reader's convenience, we  keep the same notations as in the optimal stopping case with some adaptations of the definitions.\vspace{10pt}

Let $U$ be the set of flows of measures on $\bar{\mathcal{O}}\times A$, $\left(m_{t}\right)_{t\in [0, T]}$, such that: for every $t \in[0, T]$, $m_{t}$ is a Borel finite signed measure on $\bar{\mathcal{O}}\times A$, for every $B \in \mathcal{B}(\bar{\mathcal{O}}\times A)$, the mapping $t \mapsto m_{t}(B)$ is measurable, and $\int_{0}^{T} |m_{t}|(\bar{\mathcal{O}}\times A) dt < \infty$, where $|m_{t}|$ is the total variation measure of $m_{t}$. The definitions of $\tilde U$, $\tilde U_p$, $V$ and $V_p$ from the previous section are adapted in a similar way. The topology $\tau_p$ denotes, as in the previous section, the weak topology with respect to continuous functions with $p$-growth. The topology $\bar \tau_p$ stands for the stable topology in $V_p$ where test functions of $(t, x, a)$ are allowed to be only measurable with respect to $t$ and have $p$-growth.\vspace{10pt}

\noindent By the disintegration theorem, for each $(m_t)_{t\in[0,T]} \in V$, there exists a mapping $\nu_{t,x}:[0,T]\times \bar{\mathcal{O}} \to \mathcal P(A)$ such that for each $B\in \mathcal B(A)$, the function $(t, x)\mapsto \nu_{t, x}(B)$ is $\mathcal{B}([0, T]\times \bar{\mathcal{O}})$-measurable, and 
$$m_t(dx, da)dt=\nu_{t, x}(da)m_t^x(dx)dt,$$
where $m_t^x(dx):=\int_A m_t(dx, da)$. Here $\nu_{t, x}$ is interpreted as a Markovian relaxed control and in the case when $\nu_{t, x}$ is a dirac mass, it is called Markovian strict control (see \cite{dlt2021} for more details). \vspace{10pt}

\noindent We define the parabolic boundary as the set $\Sigma=([0, T)\times \partial \mathcal{O}) \cup (\{T\}\times \bar{\mathcal{O}})$. We are given constants $q> p\geq 1\vee r$, where $r\in [0, 2]$ and $q\geq 2$, and the following functions:
$$(b, \sigma):[0, T]\times \mathbb R\times A\rightarrow \mathbb R,\quad f:[0, T]\times \bar{\mathcal{O}}\times \mathcal{P}_p^{sub}(\bar{\mathcal{O}})\times A\rightarrow \mathbb R,\quad g:\Sigma \times \mathcal{P}_p(\Sigma)\rightarrow \mathbb R.$$

Consider the following assumptions, under which existence of LP MFG Nash equilibria can be shown.

\begin{assumption}\label{assump existence SC}\leavevmode
\begin{enumerate}[(1)]
\item $m_0^*\in \mathcal{P}_q(\bar{\mathcal{O}})$.

\item $\mathcal{O}$ is a bounded open interval, $\sigma$ does not depend on the control $a$ and for all $(t, x)\in [0, T]\times \mathbb R$, $\sigma^2(t, x)\geq c_\sigma$ for some $c_\sigma>0$. 

\item The functions $(t, x, a)\mapsto b(t, x, a)$ and $(t, x)\mapsto\sigma(t, x)$ are jointly measurable and continuous in $(x, a)$ and $x$ respectively, for each $t$. Moreover, there exists a constant $c_1>0$ such that for all $(t, x, y, a)\in [0, T]\times \mathbb R\times \mathbb R\times A$,
$$|b(t, x, a)-b(t, y, a)|+|\sigma(t, x)-\sigma(t, y)|\leq c_1 |x-y|,$$
$$|b(t, x, a)|\leq c_1\left[1+|x|\right],\quad \sigma^2(t, x)\leq c_1\left[ 1+|x|^r\right].$$

\item The function $(t, x, \eta, a)\mapsto f(t, x, \eta, a)$ is jointly measurable and continuous in $(x, \eta, a)$ for each $t$. The function $g$ is continuous. Moreover, there exists a constant $c_2>0$ such that for all $(t, x, \eta, a, (\bar t, \bar x), \mu)\in [0, T]\times \bar{\mathcal{O}}\times \mathcal{P}_p^{sub}(\bar{\mathcal{O}})\times A \times \Sigma \times \mathcal{P}_p(\Sigma)$,
$$|f(t, x, \eta, a)|\leq c_2\left[ 1+|x|^p + \int_{\bar{\mathcal{O}}}|z|^p\eta(dz) \right],\quad |g(\bar t, \bar x, \mu)|\leq c_2\left[ 1+|\bar x|^p + \int_{\Sigma}|z|^p\mu(ds, dz)\right].$$
\end{enumerate}
\end{assumption}

We give below the linear programming formulation of the MFG problem with pure control and absorption, which has been introduced in \cite{dlt2021}.

\begin{definition}
Let $\mathcal{R}$ be the set of pairs $(\mu, m) \in \mathcal{P}_p(\Sigma)\times V_p$, such that for all $u\in C_b^{1, 2}([0, T]\times \bar{\mathcal{O}})$,
\begin{equation*}
\int_{\Sigma} u(t, x)\mu(dt, dx)= \int_\mathcal{O} u(0, x)m_0^*(dx) + \int_0^T \int_{\bar{\mathcal{O}}\times A} \left(\partial_t u +\mathcal L u\right) (t, x, a)m_t(dx, da)dt,
\end{equation*}
where
$$\mathcal L u(t, x, a):= b(t, x, a)\partial_x u(t, x) + \frac{\sigma^2}{2}(t, x)\partial_{xx} u(t, x).$$
\end{definition}

\begin{definition}[\textit{LP formulation of the MFG problem}]
For $(\bar \mu, \bar m)\in \mathcal{P}_p(\Sigma)\times V_p$, let $\Gamma[\bar \mu, \bar m]: \mathcal{P}_p(\Sigma)\times V_p\rightarrow \mathbb R$ be defined by
$$\Gamma[\bar \mu, \bar m] (\mu, m)= \int_0^T \int_{\bar{\mathcal{O}}\times A} f(t, x, \bar m_t^x, a) m_t(dx, da)dt + \int_{\Sigma} g(t, x, \bar \mu) \mu (dt, dx).$$
We say that $(\mu^\star, m^\star)\in \mathcal{P}_p(\Sigma)\times V_p$ is an LP MFG Nash equilibrium if $(\mu^\star, m^\star)\in \mathcal{R}$ and for all $(\mu, m)\in \mathcal{R}$, $\Gamma[\mu^\star, m^\star] (\mu, m)\leq \Gamma[\mu^\star, m^\star] (\mu^\star, m^\star)$.
\end{definition}

\noindent Observe that, for a given $(\bar{\mu},\bar{m}) \in \mathcal{P}_p(\Sigma)\times V_p$, the instantaneous reward function $f$ only depends on the marginal in $x$ of $\bar{m}$, i.e. $\bar m^x$. Henceforth, $f$ does not depend on the distribution of the controls. Note that in this setting, due to the absorption feature of the game, the flow of measures $(m_t)$ is not necessarily continuous. Under Assumption \ref{assump existence SC}, one can show existence of an LP MFG Nash equilibrium by using the tools developed in the section on optimal stopping, which allow to prove existence in a much more general framework than in \cite{dlt2021} (in particular, $f$, $g$, $b$ and $\sigma$ might have polynomial growth with respect to $(x,m,\mu)$ and  the reward maps $f$ and $g$ are allowed to have a general dependence on the measures $m$, resp. $\mu$).

\begin{assumption}\label{main assump SC}
We assume the following:
\begin{enumerate}[(1)]
\item  There exist functions $f_1:[0, T]\times \bar{\mathcal{O}}\times \mathcal{P}_p^{sub}(\bar{\mathcal{O}})\rightarrow\mathbb R$ and $f_2:[0, T]\times \bar{\mathcal{O}}\times A\rightarrow\mathbb R$ satisfying the same conditions as $f$ and such that $f=f_1 +f_2$.
\item For each $(\bar \mu, \bar m) \in \mathcal{R}$, there exists a unique maximizer of 
$\Gamma[\bar \mu, \bar m]$ on $\mathcal{R}$.
\item The Lasry–Lions monotonicity condition holds: for all $(\mu, m)$ and $(\tilde \mu, \tilde m)$ in $\mathcal{P}_p(\Sigma)\times V_p$,
\begin{equation*}
\langle f_1(m^x) - f_1(\tilde m^x),m^x-\tilde m^x\rangle 
+ \langle g(\mu)-g(\tilde \mu),\mu-\tilde \mu\rangle \leq 0.
\end{equation*}
\item There exist constants $c_f\geq 0$ and $c_g\geq 0$ such that for all $t\in [0, T]$, $x, x'\in \bar{\mathcal{O}}$, $\eta, \eta'\in \mathcal{P}^{sub}_p(\bar{\mathcal{O}})$, $(\bar t, \bar x), (\bar t', \bar x')\in \Sigma$, $\mu, \mu'\in \mathcal{P}_p(\Sigma)$, 
$$\left|f_1(t, x, \eta) - f_1(t, x, \eta')\right|\leq c_f(1+|x|) \int_{\bar{\mathcal{O}}}(1+|z|^p)|\eta-\eta'|(dz),$$
$$|f_1(t, x, \eta) - f_1(t, x, \eta') - f_1(t, x', \eta) + f_1(t, x', \eta')|\leq c_f |x-x'|\int_{\bar{\mathcal{O}}}(1+|z|^p)|\eta-\eta'|(dz),$$ 
$$|g(\bar t, \bar x, \mu) - g(\bar t, \bar x, \mu')|\leq c_g(1+|\bar x|) \int_{\Sigma}(1+|z|^p)|\mu-\mu'|(ds, dz),$$
$$|g(\bar t, \bar x, \mu) - g(\bar t, \bar x, \mu') - g(\bar t', \bar x', \mu) + g(\bar t', \bar x', \mu')|\leq c_g (|\bar t-\bar t'|+|\bar x-\bar x'|)\int_{\Sigma}(1+|z|^p)|\mu-\mu'|(ds, dz).$$
\end{enumerate}
\end{assumption}
\noindent For sufficient conditions on the coefficients under which the above assumptions hold, the reader is referred to Appendix \ref{app suff cond SC}. Using the same arguments as in Proposition \ref{uniq OS}, one can show that there exists at most one LP MFG Nash equilibrium. \vspace{10pt}

\noindent We propose the following algorithm for computing the LP MFG Nash equilibrium.\vspace{10pt}

\begin{algorithm}[H]\label{algo 2}
\SetAlgoLined
\KwData{A number of steps $N$ for the equilibrium approximation; a pair $(\bar \mu^{(0)}, \bar m^{(0)})\in \mathcal{R}$;}
\KwResult{Approximate LP MFG Nash equilibrium}
\For{$\ell=0, 1, \ldots, N-1$}{
Compute a linear programming best response $(\mu^{(\ell+1)}, m^{(\ell+1)})$ to $(\bar \mu^{(\ell)}, \bar m^{(\ell)})$ by solving the linear programming problem
$$\argmax_{(\mu, m)\in \mathcal{R}}\Gamma[\bar \mu^{(\ell)}, \bar m^{(\ell)}](\mu, m).$$
\\
Set $(\bar \mu^{(\ell+1)}, \bar m^{(\ell+1)}):=\frac{\ell}{\ell+1}(\bar \mu^{(\ell)}, \bar m^{(\ell)}) + \frac{1}{\ell+1}(\mu^{(\ell+1)},  m^{(\ell+1)})=\frac{1}{\ell + 1}\sum_{\nu=1}^{\ell+1}(\mu^{(\nu)},  m^{(\nu)})$ \\
}
\caption{LPFP algorithm (MFGs with pure control and absorption)}
\end{algorithm}
\vspace{5pt}

\noindent Using the topology of the convergence in measure (which is denoted by $\Tilde{\tau}_p$) for the marginals $m^x$ given by $m^x_t(dx)=\int_A m_t(dx, da)$  and appropriate estimates (given in terms of well chosen metrics), the convergence of the algorithm follows by similar arguments as in Theorem \ref{main theo OS}.

\begin{theorem}[\textit{Convergence of the algorithm}]
Let Assumptions \ref{assump existence SC} and \ref{main assump SC} hold true and consider the sequences $(\bar\mu^{(N)}, \bar m^{(N)})_{N\geq 1}$ and $(\mu^{(N)}, m^{(N)})_{N\geq 1}$ generated by the Algorithm \ref{algo 2}. Then both sequences converge in the product topology $\tau_p\otimes \bar{\tau}_p$ to the unique LP MFG Nash equilibrium $(\mu^\star, m^\star)$. Furthermore, the sequences  $(\bar\mu^{(N)}, \bar m^{x,(N)})_{N\geq 1}$ and $(\mu^{(N)}, m^{x, (N)})_{N\geq 1}$ converge in the product topology $\tau_p\otimes \tilde{\tau}_p$ to $(\mu^\star, m^{x,\star})$.
\end{theorem}

\paragraph{Numerical example.}

We now illustrate the LPFP algorithm for MFGs with pure control and absorption through a numerical example. In this example, let $T=1$ and assume that the state of the representative player belongs to the domain $\bar{\mathcal O}$ with $\mathcal{O}=]-2,2[$, and is given by
$$
X_t^\alpha=X_0^\alpha+\int_0^t\alpha_s ds + W_t,
$$
i.e. $b(t, x, a)=a$ and $\sigma(t, x, a)=1$. The control $\alpha$ is assumed to take values in $A=[-1, 1]$.  The initial states of the players are distributed according to the law $\mathcal{N}(0, 0.1)$ truncated to $\mathcal O$. 

Before exiting the game at time $\tau_{\mathcal O}^{X^\alpha}\wedge T$, the representative player receives an instantaneous reward
$$-10\int_{[-2, 2]}e^{-|X_t^\alpha-y|}\eta(dy)-2||X_t^\alpha|-1|-\alpha_t^2,
$$
and the terminal reward at exit time is given by $-|X_{\tau_{\mathcal O}^{X^\alpha} \wedge T}^\alpha|$, that is:
$$f(t, x, \eta, a)=-10\int_{[-2, 2]}e^{-|x-y|}\eta(dy)-2||x|-1|-a^2, \quad g(t, x, \mu)=-|x|.$$

During the game, players have an incentive to be near the points $-1$ or $1$, and converge to the point $0$ at the final time, but at the same time the mean-field dependence creates an incentive to be far from other players. 

We discretize the linear program in a similar way to the optimal stopping case. More precisely, we consider a time grid $t_i=i \Delta$ with $\Delta=\frac{T}{n_t}$, for $i\in \{0, 1, \ldots n_t\}$, a state grid $x_{j+1}=x_j+\delta$, for $j\in \{0, 1, \ldots n_s-1\}$ with $x_0\in\mathbb R$ and $\delta>0$ and an action grid $a_0<\ldots<a_{n_a}$. We define
$$Lu(t, x, a)=\frac{\partial u}{\partial t}(t, x) + b(t, x, a)\frac{\partial u}{\partial x}(t, x) + \frac{\sigma^2}{2}(t, x)\frac{\partial^2 u}{\partial x^2}(t, x), \quad \forall u\in \mathcal{D}(L):=C^{1, 2}_b([0, T]\times \bar{\mathcal{O}}).$$
We set $\mathcal{D}(\hat L)$ as the functions in $\mathcal{D}(L)=C^{1, 2}_b([0, T]\times \bar{\mathcal{O}})$ restricted to the time-state discretization grid. For $u\in \mathcal{D}(\hat L)$, we discretize the derivatives as follows
$$\hat L_t u (t_i, x_j)= \frac{1}{\Delta}[u(t_{i+1}, x_j)-u(t_i, x_j)],$$
$$\hat L_x^u u (t_i, x_j, a_k)=\frac{1}{\delta}\max(b(t_i, x_j, a_k), 0)[u(t_{i+1}, x_{j+1})-u(t_{i+1}, x_{j})],$$
$$\hat L_x^d u (t_{i}, x_j, a_k)=\frac{1}{\delta}\min(b(t_i, x_j, a_k), 0)[u(t_{i+1}, x_{j})-u(t_{i+1}, x_{j-1})],$$
$$\hat L_{xx} u (t_i, x_j)=\frac{1}{\delta^2}\frac{\sigma^2}{2}(t_i, x_j)[u(t_{i+1}, x_{j+1}) + u(t_{i+1}, x_{j-1})-2u(t_{i+1}, x_j)].$$
The discretized generator has the form:
$$\hat L u (t_i, x_j, a_k) = \hat L_t u (t_i, x_j) + \hat L_x^u u (t_i, x_j, a_k) + \hat L_x^d u (t_i, x_j, a_k) + \hat L_{xx} u (t_i, x_j).$$
The constraint reads as
\begin{multline*}
\sum_{i=0}^{n_t-1}\sum_{j\in\{0, n_s\}} u(t_i, x_j)\mu(t_i, x_j)+ \sum_{j=0}^{n_s} u(t_{n_t}, x_j)\mu(t_{n_t}, x_j) - \Delta \sum_{i=0}^{n_t-1}\sum_{j=1}^{n_s-1}\sum_{k=0}^{n_a} \hat L u(t_i, x_j, a_k) m(t_i, x_j, a_k)\\
= \sum_{j=1}^{n_s-1}u(t_0, x_j) m_0^*(x_j), 
\end{multline*}
for $u\in \mathcal{D}(\hat L)$. As in the optimal stopping case, it suffices to evaluate the constraint on the set of indicator functions. The discretized reward associated to a discrete mean-field term $(\bar\mu, \bar m)$ is given by
\begin{multline*}
\sum_{i=0}^{n_t-1}\sum_{j\in\{0, n_s\}} g(t_i, x_j, \bar\mu)\mu(t_i, x_j)+ \sum_{j=0}^{n_s} g(t_{n_t}, x_j, \bar\mu)\mu(t_{n_t}, x_j) \\
+ \Delta\times \sum_{i=0}^{n_t-1}\sum_{j=1}^{n_s-1}\sum_{k=0}^{n_a} f(t_i, x_j, \bar m^x(t_i, \cdot), a_k) m(t_i, x_j, a_k).
\end{multline*}
The generator obtained using these approximations is associated to the following controlled Markov chain (see p. 328 in \cite{kushner2001}):
$$\mathbb P(Y_{t_{i+1}}=x_{j}|Y_{t_{i}}=x_{j}, \alpha_{t_{i}}=a_k)=1-\sigma^2(t_i, x_j)\frac{\Delta}{\delta^2} - |b(t_i, x_j, a_k)|\frac{\Delta}{\delta},$$
$$\mathbb P(Y_{t_{i+1}}=x_{j+1}|Y_{t_{i}}=x_{j}, \alpha_{t_{i}}=a_k)=\frac{\sigma^2}{2}(t_i, x_j)\frac{\Delta}{\delta^2} + b^+(t_i, x_j, a_k)\frac{\Delta}{\delta},$$
$$\mathbb P(Y_{t_{i+1}}=x_{j-1}|Y_{t_{i}}=x_{j}, \alpha_{t_{i}}=a_k)=\frac{\sigma^2}{2}(t_i, x_j)\frac{\Delta}{\delta^2} + b^-(t_i, x_j, a_k)\frac{\Delta}{\delta}.$$
For this to be well defined, we should have for all $i$, $j$ and $k$
$$\Delta\leq \frac{\delta^2}{\sigma^2(t_i, x_j) + \delta |b|(t_i, x_j, a_k)}.$$
The discretized constraint coincides with the constraint associated to the controlled Markov chain $Y$ (with absorption on $\{x_0, x_{n_s}\}$).\vspace{5pt}

\noindent At each iteration the exploitability writes: 
\begin{multline*}
\varepsilon_N=\Delta \sum_{i=0}^{n_t-1}\sum_{j=1}^{n_s-1}\sum_{k=0}^{n_a}f(t_i, x_j, \bar m^{x, (N-1)}(t_i, \cdot), a_k)(m^{(N)}(t_i, x_j, a_k)-\bar m^{(N-1)}(t_i, x_j, a_k)) \\ 
+ \sum_{i=0}^{n_t}\sum_{j\in\{0, n_s\}}g(t_i, x_j, \bar \mu^{(N-1)})(\mu^{(N)}(t_i, x_j)- \bar \mu^{(N-1)}(t_i, x_j)) \\
+ \sum_{j=0}^{n_s} g(t_{n_t}, x_j, \bar \mu^{(N-1)})(\mu^{(N)}(t_{n_t}, x_j)- \bar \mu^{(N-1)}(t_{n_t}, x_j)).
\end{multline*}

In Figure \ref{Distribution SC}, we observe the distribution of the players still in the game over time together with the exit distributions at the boundary and the distribution of the players at the terminal time. Figure \ref{control} shows the Markovian control given by $\bar{\alpha}^{(N)}(t, x)=\int_Aa\bar{\nu}_{t, x}^{(N)}(da)$ (which is the optimal control since $A$ is convex, $b$ is affine in $a$ and $f$ is strictly concave in $a$, see the proof of Theorem \ref{unique max sc} for more details). We see that players starting in a positive state use a positive control at the beginning to be near the point $1$ and switch to a negative control towards the end of the game to be close to $0$. On the other hand, the players starting in a negative state use the opposite strategy. Finally, Figure \ref{error_sc} illustrates the convergence of the algorithm through the measurement of the exploitability.

\begin{figure}[h!]
    \centering
    \includegraphics[width=1.\textwidth]{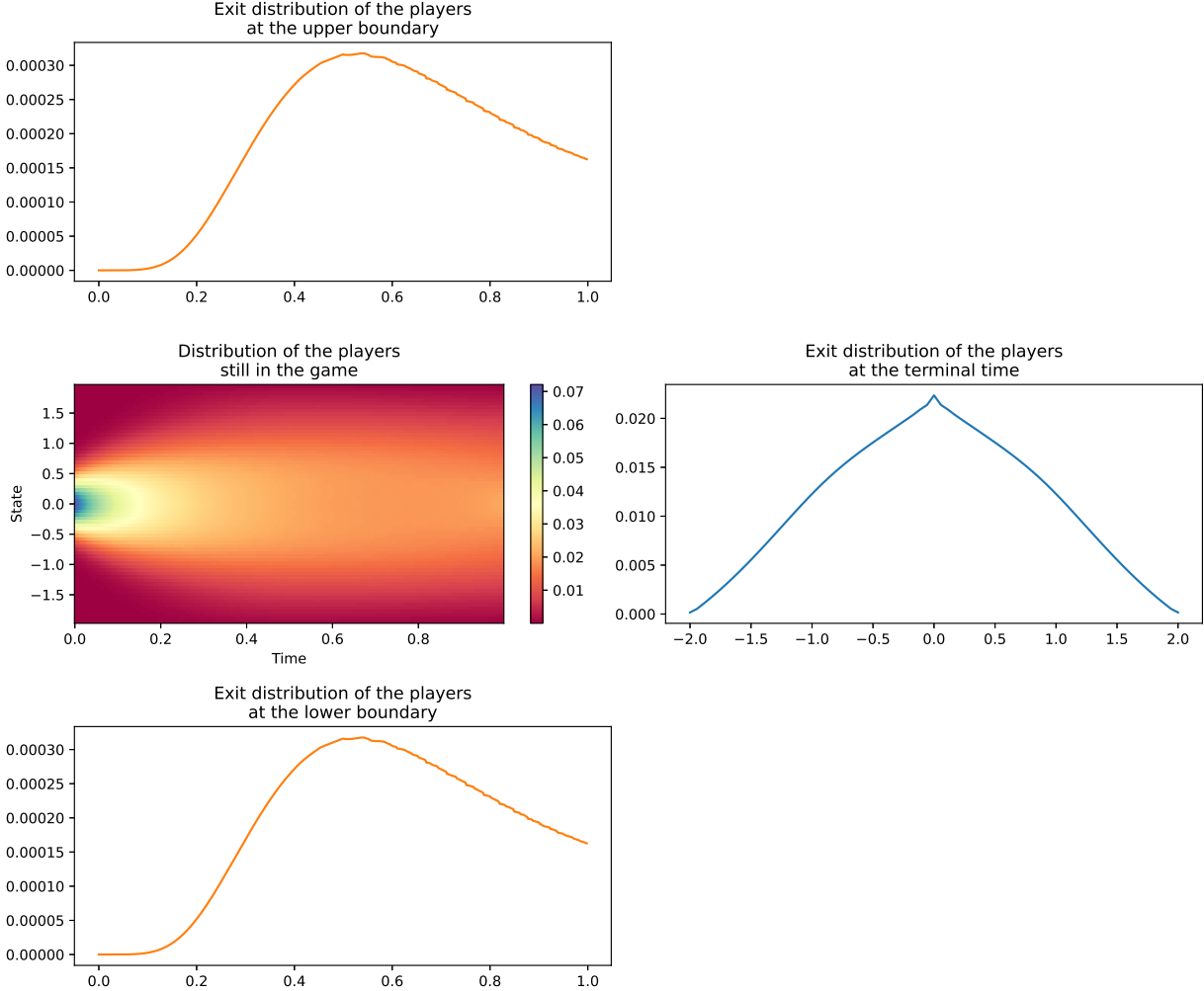}
    \caption{Equilibrium distributions at the final iteration. Top: $\bar \mu^{(N)}(\cdot\times \{2\})$. Middle left: $\bar m^{x, (N)}$, distribution of the players still in the game over time. Middle right: $\bar \mu^{(N)}(\{T\}\times \cdot)$. Bottom: $\bar \mu^{(N)}(\cdot\times \{-2\})$.}
    \label{Distribution SC}
\end{figure}

\begin{figure}
    \centering
    \includegraphics[width=0.6\textwidth]{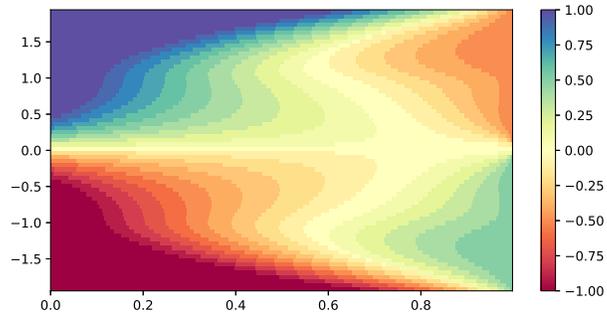}
    \caption{Markovian control at the equilibrium.}
    \label{control}
\end{figure}

\begin{figure}
    \centering
    \includegraphics[width=0.6\textwidth]{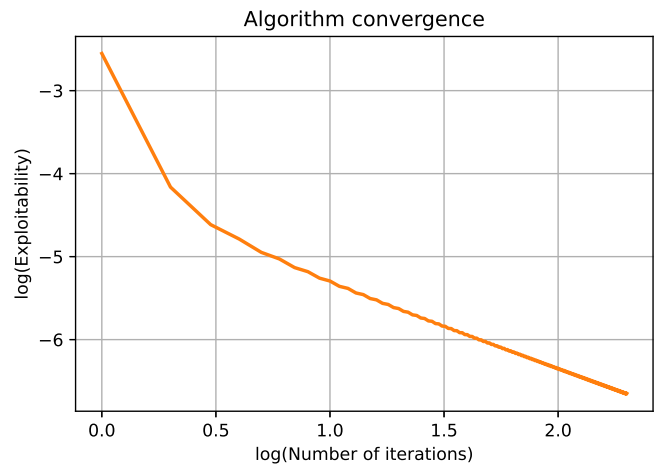}
    \caption{Convergence of the algorithm: log-log plot of the exploitability.}
    \label{error_sc}
\end{figure}

\newpage
\appendix

\section{Polynomial growth topologies for measures}\label{app tau p}

Let $(E, d)$ be a complete and separable metric space. We endow $\mathcal{M}^s(E)$ with the topology of weak convergence $\tau_0:=\sigma(\mathcal{M}^s(E), C_b(E))$. The relative topology in $\mathcal{M}(E)$ is completely metrizable through the bounded Lipschitz distance (see \cite{bogachev2007}, volume II, p.192 and Theorem 8.3.2 p.193)
$$d_{\text{BL}}(\mu^1, \mu^2):=\sup\left\{\int_E \varphi(x)(\mu^1-\mu^2)(dx): \; \varphi\in \text{BL}(E),\; \|\varphi\|_{\text{BL}} \leq 1  \right\},\quad \mu^1, \mu^2\in \mathcal{M}(E),$$
where $\text{BL}(E)$ is the space of all bounded Lipschitzian functions on $E$ with the norm
$$\|\varphi\|_{\text{BL}}:=\|\varphi\|_\infty + \sup_{x\neq y} \frac{|\varphi(x)-\varphi(y)|}{|x-y|}.$$
Let $p\geq 1$ and $x_0$ be an arbitrary point in $E$ and define the function $\psi:E\rightarrow\mathbb R$ by $\psi(x)=1+d(x, x_0)^p$. Consider the class of functions 
$$C_p(E)=\left\{\phi\in C(E): \sup_{x\in E}\frac{|\phi(x)|}{\psi(x)}<\infty\right\}$$
and define the topology $\tau_p:=\sigma(\mathcal{M}_p^s(E), C_p(E))$ on $\mathcal{M}_p^s(E)$.\vspace{10pt}

We give below some technical results, for which we do not provide the proofs since they use standard arguments (see e.g. Appendix A in \cite{lacker2015} and Theorem 1 in \cite{gomes2010}).

\begin{lemma}\label{homeo 1}
The function $F:(\mathcal{M}_p^s(E), \tau_p)\rightarrow (\mathcal{M}^s(E), \tau_0)$ given by $F(\mu)=\psi(x)\mu(dx)$ is an homeomorphism. The same function is an homeomorphism between the subspaces $(\mathcal{M}_p(E), \tau_p)$ and $(\mathcal{M}(E), \tau_0)$.
\end{lemma}

\begin{remark}
Since $(\mathcal{M}(E), \tau_0)$ is completely metrizable through the bounded Lipschitz distance, we get that $(\mathcal{M}_p(E), \tau_p)$ is also completely metrizable by the metric
$$d_{\text{BL}, p}(\mu^1, \mu^2):=\sup\left\{\int_E \varphi(x)\psi(x)(\mu^1-\mu^2)(dx): \; \varphi\in \text{BL}(E),\; \|\varphi\|_{\text{BL}} \leq 1  \right\},\quad \mu^1, \mu^2\in \mathcal{M}_p(E).$$
In particular, for all $ \mu\in \mathcal{M}_p(E)$, $d_{\text{BL}, p}(\mu, 0) = \int_E\psi(x)\mu(dx)$.
\end{remark}

The following proposition characterizes the convergence in $\tau_p$ for nonnegative measures. The proof is analogous to the one from Theorem 7.12 in \cite{villani2003}.

\begin{proposition}\label{carac p convergence}
A sequence $(\mu_n)_{n\geq 1}\subset\mathcal{M}_p(E)$ converges to $\mu\in \mathcal{M}_p(E)$ in $\tau_p$ if and only if $(\mu_n)_{n\geq 1}$ converges to $\mu$ weakly and 
\begin{equation}\label{unif int p}
\lim_{r\rightarrow\infty}\limsup_{n\rightarrow \infty}\int_{\{x\in E:\; d(x, x_0)^p\geq r\}}d(x, x_0)^p\mu_n(dx)=0.
\end{equation}
\end{proposition}

\begin{corollary}\label{compact M_p}
Assume that $E$ is a closed subset of an Euclidean space with norm $|\cdot|$. A set $\mathcal{K}\subset \mathcal{M}_p(E)$ is relatively compact in $\tau_p$ if there exists $q>p$ such that
$$\sup_{\mu\in \mathcal{K}}\int_E (1+|x|^q)\mu(dx)<\infty.$$
\end{corollary}
\begin{proof}
The condition
$$\sup_{\mu\in \mathcal{K}}\int_E (1+|x|^q)\mu(dx)<\infty,$$
implies that $\mathcal{K}$ is tight (since the map $x\mapsto 1+|x|^q$ has compact level sets) and uniformly bounded in total variation norm. By Prokhorov’s Theorem (Theorem 8.6.2 in \cite{bogachev2007} Volume II), the set $\mathcal{K}$ is relatively compact in $\tau_0$. Henceforth, by Proposition \ref{carac p convergence}, it suffices to show the uniform integrability condition \eqref{unif int p} for a given sequence $(\mu_n)_{n\geq 1}\subset\mathcal{K}$. By Hölder's inequality, for all $r\geq 0$
\begin{align*}
\int_{\{x\in E:\; |x|^p\geq r\}}|x|^p\mu_n(dx)&\leq \left[\int_E|x|^q\mu_n(dx)\right]^{p/q}\mu_n(\{x\in E:\; |x|^p\geq r\})^{(q-p)/q}\\
&\leq C\mu_n(\{x\in E:\; |x|^p\geq r\})^{(q-p)/q}.
\end{align*}
By Markov's inequality,
\begin{align*}
\mu_n(\{x\in E:\; |x|^p\geq r\})\leq \frac{1}{r}\int_E|x|^p\mu_n(dx)\leq \frac{C}{r}.
\end{align*}
This suffices to conclude.
\end{proof}

Now we are interested in an analogue version of the stable convergence topology for positive measures where the test functions are allowed to have polynomial growth. Consider two complete separable metric spaces $(E, d_E)$ and $(F, d_F)$. The distance on the product space is given by
$$d((x, y), (x', y'))=(d_E(x, x')^p+d_F(y, y')^p)^{1/p}, \quad (x, y)\in E\times F.$$
Let $(x_0, y_0)\in E\times F$ and define the function
$$\bar\psi: E\times F\ni (x, y)\mapsto 1+ d((x, y), (x_0, y_0))^p\in \mathbb R.$$
Consider the following sets of functions, which are measurable in the first component and continuous in the second one:
$$M_{mc}(E\times F)=\left\{\phi\in M_b(E\times F): \forall x\in E, \; \phi(x, \cdot)\in C(F)\right\},$$
$$M_{mc, p}(E\times F)=\left\{\phi\in M(E \times F): \forall x\in E, \; \phi(x, \cdot)\in C(F), \; \sup_{(x, y)\in E\times F}\frac{|\phi(x, y)|}{\bar \psi(x, y)}<\infty\right\}.$$

The topology $\bar\tau:=\sigma(\mathcal{M}(E\times F), M_{mc}(E\times F))$ is known as the topology of stable convergence (see \cite{jacod1981}). We are interested in studying the space $\mathcal{M}_p(E\times F)$ endowed with the topology $\bar\tau_p:=\sigma(\mathcal{M}_p(E\times F), M_{mc, p}(E\times F))$.

\begin{lemma}\label{homeo 2}
The function $\bar F:(\mathcal{M}_p(E\times F), \bar \tau_p)\rightarrow (\mathcal{M}(E\times F), \bar \tau)$ given by $\bar F(\mu)=\bar\psi(x, y)\mu(dx, dy)$ is an homeomorphism.
\end{lemma}

\begin{remark}
By Proposition 2.10 in \cite{jacod1981}, the space $(\mathcal{M}(E\times F), \bar \tau)$ is metrizable, henceforth, $(\mathcal{M}_p(E\times F), \bar \tau_p)$ is also metrizable.
\end{remark}

\begin{proposition}\label{p stable}
Consider a sequence $(\mu_n)_n\subset \mathcal{M}_p(E\times F)$ converging to $\mu\in \mathcal{M}_p(E\times F)$ in $\tau_p$. If the set of measures 
$$\left\{ \int_{\cdot\times F} \bar\psi(x, y)\mu_n(dx, dy)\in \mathcal{M}(E): n\geq 1  \right\},$$
is relatively compact in $(\mathcal{M}(E), \sigma(\mathcal{M}(E), M_b(E)))$, then $(\mu_n)_n$ converges to $\mu$ in $\bar\tau_p$.
\end{proposition}

\section{Convergence in measure topology}\label{app conv m}

Let $(E, d)$ be a complete and separable metric space endowed with the Borel $\sigma$-algebra and let $M([0, T];E)$ be the space of Borel measurable functions $\phi:[0, T]\rightarrow E$ identified a.e. on $[0, T]$. The topology of convergence in measure in $M([0, T];E)$ is defined as the topology induced by the metric (see e.g. \cite{rossi2003})
$$d_M(\phi, \psi)=\int_0^T 1\wedge d(\phi(t), \psi(t))dt.$$
A sequence $(\phi_n)_{n\geq 1}$ converges to $\phi$ in $M([0, T];E)$ if and only if for all $\varepsilon>0$
$$\lim_{n\rightarrow\infty}\lambda(\{t\in [0, T]: d(\phi_n(t), \phi(t))\geq \varepsilon\})=0.$$
We recall that convergence in $M([0, T]; E)$ implies convergence of a subsequence in $E$ $t$-a.e. on $[0, T]$. The topology of convergence in measure remains invariant with respect to any metric inducing the same topology as $d$ on $E$.

\begin{lemma}\label{bound conv in meas}
Let $M_0\subset M([0, T];E)$ and assume that there exists a constant $C\geq 0$ such that for all $\phi, \psi\in M_0$ we have $d(\phi(t), \psi(t))\leq C$, $t$-a.e. on $[0, T]$. Then
$$d_{M_0}(\phi, \psi)=\int_0^Td(\phi(t), \psi(t))dt, \quad \phi, \psi\in M_0$$
metrizes the topology of convergence in measure in $M_0$.
\end{lemma}

\begin{proof}
First note that for $(\phi_n)_{n\geq 1}\subset M_0$ and $\phi\in M_0$,
$$\lim_{n\rightarrow\infty} d_{M_0}(\phi_n, \phi)=0 \Rightarrow  \lim_{n\rightarrow\infty}d_{M}(\phi_n, \phi)=0.$$
The converse implication follows since the sequence $(d(\phi_n(\cdot), \phi(\cdot)))_{n\geq 1}$ converges in measure to $0$ and is bounded a.e. by $C$, which implies the convergence in $L^1([0, T])$ to $0$.
\end{proof}

\section{Probabilistic representation}\label{app prob rep OS}

In the case when admissible measures $m_t(dx)dt$ (resp. $\mu$) have the support included in some time-dependent domain $O$ (resp. its complement $O^c$), we obtain the following probabilistic representation.

\begin{theorem}\label{exit rep}
Suppose that Assumption \ref{assump existence OS} holds. Let $(\mu, m)\in \mathcal{R}$ and $O$ be an open subset of $[0, T]\times \bar{\mathcal{O}}$. Assume that $\mu(O)=0$ and $\int_{O^c}m_t(dx)dt=0$. By Theorem C.6 in \cite{dlt2021}, there exists $(\Omega, \mathcal{F}, \mathbb F, \mathbb P, W, \tau, X)$, such that $(\Omega, \mathcal{F}, \mathbb F, \mathbb P)$ is a filtered probability space, $W$ is an $\mathbb F$-Brownian motion, $\tau$ is an $\mathbb F$-stopping time such that $\tau\leq T\wedge\tau_\mathcal{O}^{X}$ $\mathbb P$-a.s.~and $X$ is an $\mathbb F$-adapted process verifying
$$X_{t}= X_0 + \int_0^{t}b(s, X_s)ds + \int_0^{t} \sigma(s, X_s)dW_s,\quad t\in [0, T], \quad \mathbb P \circ X_0^{-1}= m_0^*,$$
such that we have the following probabilistic representation of $(\mu, m)$:
$$\mu =\mathbb P \circ (\tau, X_\tau)^{-1},\quad \text{and}\quad m_t(B) = \mathbb E^{\mathbb P}\left[ \mathds{1}_B(X_t)\mathds{1}_{t< \tau}\right],  \quad B\in \mathcal{B}(\bar{\mathcal{O}}), \quad t-a.e.$$
Assume that $\tau_O=\tau_{\bar{O}}$ $\mathbb P$-a.s., 
where
$$\tau_O=\inf\{t\geq 0: (t, X_t)\notin O \}, \quad \tau_{\bar{O}}=\inf\{t\geq 0: (t, X_t)\notin \bar{O} \}.$$
Then $\tau=\tau_O$ $\mathbb P$-a.s.
\end{theorem}

\begin{proof}
Let us show that $\tau=\tau_O$ $\mathbb P$-a.s. Using that $\mu$ is supported in $O^c$ we get
\begin{equation}\label{exit 1}
1=\mu(O^c)=\mathbb P((\tau, X_\tau)\in O^c).    
\end{equation}
Now, since $m_t(dx)dt$ is supported in $O$,
\begin{align*}
0=\int_{O^c}m_t(dx)dt=\mathbb E^{\mathbb P}\left[ \int_0^\tau \mathds{1}_{O^c}(t, X_t) dt \right], 
\end{align*}
which means that
\begin{equation}\label{exit 2}
(\mathbb P \otimes \lambda)(\{(\omega, t)\in \Omega\times [0, T]: (t, X_t(\omega))\in O^c, \; t< \tau(\omega)\})=0.    
\end{equation}
By equality \eqref{exit 1}, we have that $\tau_O\leq \tau$ $\mathbb P$-a.s. Let us show now that with probability $1$ we have $\tau\leq \tau_{\bar O}$, where $\tau_{\bar O}=\inf\{t\geq 0: (t, X_t)\notin \bar O \}$. Assume that $\mathbb P(\tau>\tau_{\bar O})>0$. Using the equality \eqref{exit 2},
$$\mathbb P(\lambda(\{t\in [0, T]: (t, X_t)\notin O, t<\tau\})=0)=1.$$
Define the set
$$B=\{\tau>\tau_{\bar O}\}\cap \{\lambda(\{t\in [0, T]: (t, X_t)\notin O, t<\tau\})=0\}\in\mathcal{F}.$$
Let us show that $B=\emptyset$, which will contradict the fact that $\mathbb P(B)>0$. It is sufficient to prove that for $\omega\in \{\tau>\tau_{\bar O}\}$, we have $\omega\notin \{\lambda(\{t\in [0, T]: (t, X_t)\notin O, t<\tau\})=0\}$. Let $\omega\in \{\tau>\tau_{\bar O}\}$ be fixed. Since $t\mapsto (t, X_t(\omega))$ is continuous, we have $(\tau_{\bar O}(\omega), X_{\tau_{\bar O}(\omega)}(\omega))\in \bar O$. Moreover, by the definition of the infimum there exists $\tau_1(\omega)\in ]\tau_{\bar O}(\omega), \tau(\omega)[$ such that $(\tau_1(\omega), X_{\tau_1(\omega)}(\omega))\notin \bar O$. Using again the continuity of $t\mapsto (t, X_t(\omega))$ and the fact that $({\bar O})^c$ is open, we can find $\tau_1^l(\omega)$ and $\tau_1^r(\omega)$ such that
$$\tau_{\bar O}(\omega)<\tau_1^l(\omega)< \tau_1(\omega)< \tau_1^r(\omega)<\tau(\omega) \quad \text{and}\quad \forall t\in ]\tau_1^l(\omega), \tau_1^r(\omega)[, \quad (t, X_t(\omega))\notin \bar O.$$
In particular,
$$\lambda(\{t\in [0, T]: (t, X_t(\omega))\notin O, t<\tau(\omega)\})\geq \lambda(]\tau_1^l(\omega), \tau_1^r(\omega)[)>0.$$
This shows that $B$ is empty and contradicts the fact that it has a positive probability. We conclude that $\mathbb P$-a.s. $\tau_O\leq \tau\leq \tau_{\bar O}$. Since by assumption we have $\tau_O=\tau_{\bar O}$ $\mathbb P$-a.s., we get $\tau=\tau_O$ $\mathbb P$-a.s.
\end{proof}

\section{Sufficient conditions for Assumption \ref{main assump OS}}\label{app suff cond OS}

In this section, we provide sufficient conditions on $b$, $\sigma$, $f$, $g$, $m_0^*$ and $\mathcal{O}$ such that Assumption \ref{main assump OS} is verified.

\begin{assumption}\label{suff conditions assump OS}
We assume the following:
\begin{enumerate}[(1)]
\item $\mathcal{O}=\mathbb R$.
\item $m_0^*$ has a continuous and positive density on $L^2(\mathbb R)$.
\item $\sigma$ is positive, does not depend on time and satisfies the uniform ellipticity condition. Moreover $\sigma\in C^{1}(\mathbb R)$,
$$\sup_{(t, x)\in [0, T]\times \mathbb R}\left[\frac{|b(t, x)|}{\sigma(x)} + |\partial_x\sigma(x)|\right]<\infty.$$
\item $f$ is of the form $f(t,x,m):=\bar{f}(t,x)$ and there exists $\bar c_f\geq 0$ such that for all $t\in [0, T]$, $x, x'\in \mathbb R$,
$$|\bar{f}(t, x)-\bar{f}(t, x')|\leq \bar c_f|x-x'|, \quad |\bar{f}(t, x)|\leq \bar c_f(1+|x|).$$
To be consistent with the notations used in Section \ref{sec OS}, we will keep the notation $f$ instead of $\bar{f}$.
\item $g$ has the form 
$$g(t, x, \mu)= g_1(t, x) g_2\left(\int_{[0, T]\times \mathbb R}g_1(s, y)\mu(ds, dy)\right) + g_3(t, x),$$
where $g_2$ is non-increasing. There exists $\bar c_g\geq 0$ and $\beta\in ]0, 1]$ such that for all $t, t'\in [0, T]$, $x, x'\in \mathbb R$,
$$|g_1(t, x)-g_1(t', x')|\leq \bar c_g(|t-t'| + |x-x'|), \quad |g_1(t, x)|\leq \bar c_g (1+|x|),$$
$$|g_2(x)-g_2(x')|\leq \bar c_g|x-x'|, \quad |g_2(x)|\leq \bar c_g (1+|x|),$$
$$|g_3(t, x)-g_3(t', x')|\leq \bar c_g(|t-t'|^\beta + |x-x'|), \quad |g_3(t, x)|\leq \bar c_g (1+|x|).$$

\item We assume
$g_1\in C^{1, 2}([0, T]\times \mathbb R)$, $g_3\in C^{1, 2}([0, T]\times \mathbb R)$, $\partial_xg_1\in C_b([0, T]\times \mathbb R)$ and $\partial_xg_3\in C_b([0, T]\times \mathbb R)$. Moreover, for each $t\in [0, T]$ and $\mu\in \mathcal{P}_p([0, T]\times \mathbb R)$,
$$x\mapsto \left(f + \partial_t g + \mathcal{L}g\right)(t, x, \mu)$$
is increasing. Finally, for each $t\leq t'$, $x\leq x'$ and $\mu\in \mathcal{P}_p([0, T]\times \mathbb R)$,
$$\left(f + \partial_t g +\mathcal{L}g\right)\left(t', x, \mu\right)\leq \left(f + \partial_t g +\mathcal{L}g\right)\left(t, x', \mu\right),\quad b(t', x)\leq b(t, x).$$
\end{enumerate}
\end{assumption}

The second and third condition of Assumption \ref{main assump OS} are easily verified under the above conditions. In the next Theorem, we prove that the first condition is also satisfied.

\begin{theorem}
Under the Assumptions \ref{assump existence OS} and \ref{suff conditions assump OS}, for each $(\bar \mu, \bar m) \in \mathcal{R}$, there exists a unique maximizer of 
$\Gamma[\bar \mu, \bar m]$ on $\mathcal{R}$.
\end{theorem}

\begin{proof}
Fix $(\bar \mu, \bar m)\in \mathcal{R}$. Let us show that there exists a unique maximizer of $\Gamma[\bar \mu, \bar m]$ on $\mathcal{R}$. Under our assumptions, this map writes
$$\Gamma[\bar \mu, \bar m](\mu, m)=\int_0^T\int_\mathbb R f(t, x)m_t(dx)dt +\int_{[0, T]\times \mathbb R}g(t, x, \bar \mu)\mu(dt, dx).$$
We characterize the maximizer via a probabilistic approach, which allows to deduce the uniqueness result. Consider a complete probability space $(\Omega, \mathcal{F}, \mathbb P)$ supporting a Brownian motion $W=(W_t)_{t\in [0, T]}$. Given $t\in [0, T]$, we denote by $\mathbb{F}^t$ the filtration given by $\mathcal{F}_{s}^t=\sigma\left(W_r^t, t \leq r \leq s\right) \vee \mathcal{N}$, $s\geq t$. Here $\mathcal{N}$ is the set of $\mathbb P$-null sets and $W^t_s=W_s-W_t$, $s\geq t$, is the translated Brownian motion. Denote by $\mathcal{T}_t$ the set of stopping times with respect to this filtration with values in $[t, T]$. Consider the value function (to simplify the notation, we omit the dependence on $\bar \mu$)
\begin{align}\label{value_fct OS}
v(t, x)=\sup _{\tau \in \mathcal{T}_t} \mathbb{E}\left[\int_{t}^{\tau} f\left(s, X_{s}^{t, x}\right) ds + g(\tau, X_\tau^{t, x}, \bar \mu)\right].
\end{align}
where $X^{t, x}$ is the unique strong solution to the SDE
$$dX_s=b(s, X_s)ds + \sigma(X_s)dW_s;\,\, X_t=x.$$
We denote by $\mathcal{C}=\{(t, x)\in [0, T]\times \mathbb R: v(t, x)>g(t, x, \bar \mu)\}$ the continuation region, by $\bar{\mathcal{C}}$ its closure and by $\mathcal{S}=\{(t, x)\in [0, T]\times \mathbb R: v(t, x)=g(t, x, \bar \mu)\}$ the stopping region. We divide now the proof into several steps.\vspace{5pt}

\noindent\textit{First Step: Properties of the value function and optimal stopping boundary.} Under our assumptions, it can be shown by using standard arguments that the value function $v$ is jointly continuous. In particular, we deduce that the continuation region $\mathcal{C}$ is an open subset of $[0, T]\times \mathbb R$ and that the stopping region $\mathcal{S}$ is a closed subset of $[0, T]\times \mathbb R$.\vspace{5pt}

Consider $x, y\in \mathbb R$ such that $x\leq y$ and let $t\in [0, T]$. By the comparison theorem for SDEs, we get $X_s^{t, x}\leq X_s^{t, y}$,  $s\in [t, T]$. For any $\tau\in \mathcal{T}_t$, using that $z\mapsto \left(f + \partial_t g +\mathcal{L}g\right)(s, z, \bar \mu)$ is increasing for each $s\in [0, T]$, a direct application of Itô's formula to $g$ gives
\begin{multline*}
\mathbb{E}\left[\int_{t}^{\tau} f\left(s, X_{s}^{t, x}\right) ds + g(\tau, X_\tau^{t, x}, \bar\mu)\right] -g(t, x, \bar\mu)\\
\leq \mathbb{E}\left[\int_{t}^{\tau} f\left(s, X_{s}^{t, y}\right) ds + g(\tau, X_\tau^{t, y}, \bar\mu)\right]-g(t, y, \bar\mu).    
\end{multline*}
Taking the supremum over $\tau\in \mathcal{T}_t$, we get that, for each $t\in [0, T]$, the function 
$x\mapsto v(t, x)-g(t, x, \bar \mu)$ 
is nondecreasing. Therefore we can define an extended real-valued function $c:[0, T]\rightarrow [-\infty, \infty]$ by $c(t)=
\inf\{x\in\mathbb R:v(t, x)>g(t, x, \bar \mu)\}$, if the set $\{x\in\mathbb R:v(t, x)>g(t, x, \bar \mu)\}$ is non-empty and lower bounded, $c(t)=-\infty$ if the set $\{x\in\mathbb R:v(t, x)>g(t, x, \bar \mu)\}$ is non-empty and not lower bounded and $c(t)=\infty$ if the set $\{x\in\mathbb R:v(t, x)>g(t, x, \bar \mu)\}$ is empty. Using the continuity of $v$ and $g$, we have that for each $t$, $\{x\in\mathbb R:v(t, x)>g(t, x, \bar \mu)\}=]c(t), \infty[$, allowing to deduce that
$$\mathcal{C}=\{(t, x)\in [0, T]\times \mathbb R: x> c(t)\}, \quad \mathcal{S}=\{(t, x)\in [0, T]\times \mathbb R: x\leq  c(t)\}.$$
Since $\mathcal{S}$ is the hypograph of $c$ and it is a closed set, we deduce that $c$ is upper semicontinuous. \vspace{5pt}

Consider the filtration $\mathbb G^t=(\mathcal{G}^t_s)_{s\in [0, T-t]}$ given by $\mathcal{G}^t_s=\mathcal{F}^t_{t+s}, \quad s\in [0, T-t]$. We denote by $\mathcal{S}_t$ the set of stopping times with respect to this filtration with values in $[0, T-t]$. One can show that
\begin{equation}\label{value t mon}
v(t, x)=\sup_{\tau\in \mathcal{S}_t}\mathbb E\left[\int_0^\tau f(t+s, X_{t+s}^{t, x})ds + g(t + \tau, X_{\tau + t}^{t, x}, \bar \mu)\right].    
\end{equation}
Observe that, for a fixed time $t$, the stopping times in $\mathcal{S}_t$ and the process $X^{t, x}$ are functionals of the translated Brownian motion $(W^t_s)_{s\in [0, T-t]}:=(W_{t+s}-W_t)_{s\in [0, T-t]}$. Since $(W^t_s)_{s\in [0, T-t]}$ and $(W_s)_{s\in [0, T-t]}$ have the same law,  without  loss of generality we can assume that the stopping times in $\mathcal{S}_t$ are with respect to the (completed) filtration of $W$ (we keep the same notation) and replace $X^{t, x}_{t+\cdot}$ in \eqref{value t mon} by the process $Y^{0,x}[t]$ which follows the dynamics
$$Y_{s}^{0,x}[t]=x+ \int_0^s b(t+u, Y_{u}^{0, x}[t])du + \int_0^s\sigma(Y_{s}^{0, x}[t])dW_u.$$
We have $\mathcal{S}_{t'}\subset \mathcal{S}_t$ for $t\leq t'$. Consider $t, t'\in [0, T]$ such that $t\leq t'$ and let $x\in \mathbb R$, we are going to show that $v(t', x) - g(t', x, \bar\mu) \leq v(t, x) - g(t, x, \bar\mu)$. By the comparison theorem for SDEs, we get
$$\mathbb P\left( Y_{s}^{0,x}[t]\geq Y_{s}^{0,x}[t'], \quad \forall s\in [0, T-t'] \right)=1.$$
In fact, let $b_1(s, x)=b(t +s, x)$ and $b_2(s, x)=b(t' +s, x)$. For all $(s, x)\in [0, T-t']\times \mathbb R$ we have $b_1(s, x)\geq b_2(s, x)$.
Let $\tau'$ be an optimal stopping time for 
$$v(t', x) = \sup_{\tau\in \mathcal{S}_{t'}}\mathbb E\left[\int_0^\tau f(t'+s, Y_{s}^{0,x}[t'])ds + g(t' + \tau, Y_{\tau}^{0,x}[t'], \bar \mu)\right].$$
Since $\mathcal{S}_{t'}\subset\mathcal{S}_t$, we get $\tau'\in \mathcal{S}_t$, henceforth 
\begin{align*}
&(v-g)(t', x, \bar\mu) - (v-g)(t, x, \bar\mu)\\
&\leq \mathbb{E}\left[\int_{0}^{\tau'} \left[\left(f + \partial_t g +\mathcal{L}g\right)\left(t'+s, Y_{s}^{0,x}[t'], \bar \mu\right) - \left(f + \partial_t g +\mathcal{L}g\right)\left(t+s, Y_{s}^{0,x}[t], \bar \mu\right)\right] ds\right] \\
&\leq 0.
\end{align*}
We deduce that $c$ is non-decreasing.\vspace{5pt}

Since $c$ is upper semicontinuous and non-decreasing, it is right-continuous. In particular, the set $D$ of discontinuities of $c$ is countable. The optimal stopping boundary writes
$$\partial\mathcal{C}=\{(t, x)\in [0, T]\times \mathbb R:x=c(t)\}\cup \bigcup_{t\in D}(\{t\}\times[c(t-), c(t)]).$$
Since the set of discontinuities of $c$ is countable and denoting by $\lambda_2$ the Lebesgue measure in $\mathbb R^2$, we get
\begin{align*}
\lambda_2(\partial \mathcal{C})&\leq \lambda_2(\{x=c(t)\})+\sum_{t\in D}\lambda_2(\{t\}\times[c(t-), c(t)])= \int_{[0, T]\times \mathbb R}\mathds{1}_{x=c(t)}\lambda_2(dt, dx)\\
&= \int_0^T \lambda(\{c(t)\})\mathds{1}_{c(t)\in \mathbb R}dt =0.
\end{align*}

\vspace{5pt}

\noindent
\textit{Second Step: Properties of the maximizers with respect to the continuation and stopping regions.} Let $(\mu, m)$ be a maximizer of $\Gamma[\bar \mu, \bar m]$ in $\mathcal{R}$. We show that $\int_\mathcal{S}m_t(dx)dt$ and $\mu(\mathcal{C})=0$. By Theorem C.6 in \cite{dlt2021}, there exist a filtered probability space $(\tilde \Omega, \tilde{\mathcal{F}}, \tilde{\mathbb F}, \tilde{\mathbb P})$, an $\tilde{\mathbb F}$-adapted process $\tilde{X}$, an $\tilde{\mathbb F}$-stopping time $\tilde \tau$, and an $\tilde{\mathbb F}$-Brownian motion $\tilde W$, such that
$$\tilde X_{t}= \tilde X_0 + \int_0^{t} b(s, \tilde X_s)ds + \int_0^{t} \sigma(\tilde X_s)d\tilde W_s, \quad \tilde {\mathbb P} \circ \tilde X_0^{-1}= m_0^*,$$
$$\mu =\tilde{\mathbb P} \circ (\tilde\tau, \tilde X_{\tilde\tau})^{-1},\quad m_t(B)= \mathbb E^{\tilde{\mathbb P}}\left[ \mathds{1}_B(\tilde X_t)\mathds{1}_{t< \tilde \tau}\right],  \quad B\in \mathcal{B}(\mathbb R), \quad t-a.e.$$
Consider the optimal stopping problem in this probabilistic set up:
\begin{equation}\label{OS tilde}
\tilde v(t, x)=\sup _{\tau \in \mathcal{T}_t^{\tilde{\mathbb F}}} \mathbb{E}^{\tilde{\mathbb P}}\left[\int_{t}^{\tau} f(s, \tilde X_{s}^{t, x}) ds + g(\tau, \tilde X_\tau^{t, x}, \bar \mu)\right],    
\end{equation}
where $\tilde X_{s}^{t, x}$ satisfies the SDE
\begin{equation}\label{sde tilde}
d\tilde X_{s} = b(s, \tilde X_{s})ds + \sigma(\tilde X_{s})d\tilde W_s, \quad \tilde X_{t}^{t, x}=x,
\end{equation}
and $\mathcal{T}_t^{\tilde{\mathbb F}}$ is the set of stopping times with respect to $\tilde{\mathbb F}$ and valued in $[t, T]$. By Chapter I, Section 2, Corollary 2.9 in \cite{peskir2006}, the optimal stopping time is given by
$$\tilde \tau_{\tilde {\mathcal{C}}}^{t, x}:=\inf\{s\in [t, T]:\tilde v(s, \tilde X_s^{t, x})=g(s, \tilde X_s^{t, x}, \bar \mu)\}.$$
By uniqueness in law of the solution of the SDE \eqref{sde tilde}, we get $v(t, x)=\tilde v(t, x)$, for all $(t, x)\in [0, T]\times \mathbb R$. Consider the optimal stopping problem \eqref{OS tilde} at time $0$, assuming that $\tilde{\mathbb P} \circ \tilde X_0^{-1}= m_0^*$: 
\begin{equation}\label{opt stop m0}
\sup _{\tau \in \mathcal{T}_0^{\tilde{\mathbb F}}} \mathbb{E}^{\tilde{\mathbb P}}\left[\int_{0}^{\tau} f(s, \tilde X_{s}) ds + g(\tau, \tilde X_\tau, \bar \mu) \right].
\end{equation}
By measurability arguments (involving the optimal stopping time which is a measurable function of $\tilde X$), we get that 
$$\int_{\mathbb R}v(0, x)m_0^*(dx)=\mathbb E^{\tilde{\mathbb P}}[v(0, \tilde X_0)]= \sup _{\tau \in \mathcal{T}_0^{\tilde{\mathbb F}}} \mathbb{E}^{\tilde{\mathbb P}}\left[\int_{0}^{\tau} f(s, \tilde X_{s}) ds + g(\tau, \tilde X_\tau, \bar \mu)\right].$$
By Theorem 2.21 in \cite{dlt2021}, we get $\Gamma[\bar \mu, \bar m](\mu, m)=\int_\mathbb R v(0, x)m_0^*(dx)$. In particular, using the probabilistic representation of $(\mu, m)$, we deduce that $\tilde \tau$ is an optimal stopping time for \eqref{opt stop m0}. Define the following processes
$$U_t:=v(t, \tilde X_t) + \int_0^t f(s, \tilde X_s)ds,$$

$$Z_t:=\int_0^t \left(f + \partial_t g +\mathcal{L}g\right)(s, \tilde X_s, \bar \mu)ds,$$ and $$\tilde M_t := \int_0^t\sigma(\tilde X_s)\partial_x g(s, \tilde X_s, \bar\mu)d\tilde W_s.$$
Observe that, since $U$ is the Snell envelope of the process $\left(\int_0^tf(s, \tilde X_s)ds+g(t, \tilde X_t, \bar\mu)\right)_{t\in [0, T]}$, by the Doob-Meyer decomposition we get $U_t=M_t-A_t$, where $M$ is an $\tilde{\mathbb F}$-martingale and $A$ is a non-decreasing $\tilde{\mathbb F}$-predictable process with $A_0=0$. By Theorem D.13 in \cite{karatzas1998b}, since 
$$\left(\int_0^t f(s, \tilde X_s)ds + g(t, \tilde X_t, \bar \mu)\right)_{t\in [0, T]}$$
is a continuous process, we get that $A$ is continuous (and in particular $M$ is continuous) and
$$\int_0^T\mathds{1}_{\zeta_t>0}dA_t=0, \quad \tilde{\mathbb P}-a.s.,$$
where $\zeta_t:=v(t, \tilde X_t)-g(t, \tilde X_t, \tilde \mu)\geq 0$. By Tanaka's formula,
$$\zeta_t=\max(\zeta_t, 0)=\zeta_0 + \int_0^t\mathds{1}_{\zeta_s>0}d\zeta_s + \frac{1}{2}L^0_t(\zeta),$$
where $L^0(\zeta)$ is the local time of $\zeta$ in $0$. We deduce that
\begin{align*}
\zeta_t &= \zeta_0 + \int_0^t\mathds{1}_{\zeta_s>0}d(U_s-Z_s- \tilde M_s) + \frac{1}{2}L^0_t(\zeta)\\
&=  \zeta_0 + \int_0^t\mathds{1}_{\zeta_s>0}dM_s - \int_0^t\mathds{1}_{\zeta_s>0}d\tilde M_s - \int_0^t\mathds{1}_{\zeta_s>0}dA_s - \int_0^t\mathds{1}_{\zeta_s>0}dZ_s + \frac{1}{2}L^0_t(\zeta)\\
&=  \zeta_0 + \int_0^t\mathds{1}_{\zeta_s>0}dM_s - \int_0^t\mathds{1}_{\zeta_s>0}d\tilde M_s  - \int_0^t\mathds{1}_{\zeta_s>0}dZ_s + \frac{1}{2}L^0_t(\zeta).
\end{align*}
We finally get
\begin{align*}
U_t &=U_0 + Z_t + \tilde M_t + \int_0^t\mathds{1}_{\zeta_s>0}dM_s - \int_0^t\mathds{1}_{\zeta_s>0}d\tilde M_s - \int_0^t\mathds{1}_{\zeta_s>0}dZ_s + \frac{1}{2}L^0_t(\zeta)\\
&= U_0 + \int_0^t\mathds{1}_{\zeta_s>0}dM_s + \int_0^t\mathds{1}_{\zeta_s=0}d\tilde M_s + \int_0^t\mathds{1}_{\zeta_s=0}dZ_s + \frac{1}{2}L^0_t(\zeta).
\end{align*}
Using that
$$\left(U_0 + \int_0^t\mathds{1}_{\zeta_s>0}dM_s + \int_0^t\mathds{1}_{\zeta_s=0}d\tilde M_s\right)_{t\in [0, T]}$$
is a local martingale, by continuity of the processes and uniqueness of the semimartingale decomposition, we get
$$-A_t=\int_0^t\mathds{1}_{\zeta_s=0}dZ_s + \frac{1}{2}L^0_t(\zeta).$$
From the above, we deduce that the process 
$$\int_0^\cdot\mathds{1}_{v(t, \tilde X_t)=g(t, \tilde X_t, \bar \mu)}\left(f + \partial_t g +\mathcal{L}g\right)(t, \tilde X_t, \bar \mu)dt$$
is non-increasing. Therefore $t$-a.e. on $[0, T]$,
$$\mathds{1}_{v(t, \tilde X_t)=g(t, \tilde X_t, \bar \mu)}\left(f + \partial_t g +\mathcal{L}g\right)(t, \tilde X_t, \bar \mu)\leq 0.$$
In particular, $(t, x)\mapsto \mathds{1}_{\mathcal{S}}(t, x)\left(f + \partial_t g +\mathcal{L}g\right)(t, x, \bar \mu)$ is non-positive $m_t(dx)dt$-a.e. Since $\tilde\tau$ is optimal, $A_{\tilde\tau}=0$, i.e.
$$\int_0^{\tilde\tau}\mathds{1}_{v(t, \tilde X_t)=g(t, \tilde X_t, \bar \mu)}\left(f + \partial_t g +\mathcal{L}g\right)(t, \tilde X_t, \bar \mu)dt=-\frac{1}{2}L_{\tilde\tau}^0(\zeta).$$
Since the Lebesgue measure of $\partial\mathcal{C}$ is $0$, by Theorem 6 in \cite{jacka1993}, we have that the local time $L^0(\zeta)$ is indistinguishable from $0$. Henceforth, taking the expectation in the last equality, we get
$$\int_\mathcal{S}\left(f + \partial_t g +\mathcal{L}g\right)(t, x, \bar \mu)m_t(dx)dt=0.$$
To simplify notation, denote by $\nu(dt, dx)=m_t(dx)dt$ (which is absolutely continuous with respect to the Lebesgue measure in $[0, T]\times\mathbb R$, since $\sigma$ satisfies the uniform ellipticity condition). Now there exists a $\nu$-negligible set $N$ such that for all $(t, x)\in N^c$, $\mathds{1}_{\mathcal{S}}(t, x)\left(f + \partial_t g +\mathcal{L}g\right)(t, x, \bar \mu)\leq 0$. In particular, using that for all $t\in [0, T]$, $x\mapsto \left(f + \partial_t g +\mathcal{L}g\right)(t, x, \bar \mu)$ is increasing, if $(t, x)\in \mathring{\mathcal{S}} \cap N^c$, $\left(f + \partial_t g +\mathcal{L}g\right)(t, x, \bar \mu)<0$, where $\mathring{\mathcal{S}}$ denotes the interior of $\mathcal{S}$. Since $\partial \mathcal{C}$ has Lebesgue measure $0$, then $\partial \mathcal{S}$ has Lebesgue measure $0$ (see \cite{aliprantis2007} p. 27), and we obtain
\begin{align*}
0 & = \int_\mathcal{S}\left(f + \partial_t g +\mathcal{L}g\right)(t, x, \bar \mu)\nu(dt, dx)\\
& = \int_{[0, T]\times \mathbb R}\mathds{1}_{\mathring{\mathcal{S}} \cap N^c}(t, x)\left(f + \partial_t g +\mathcal{L}g\right)(t, x, \bar \mu)\nu(dt, dx).
\end{align*}
In other words, $\nu(\mathring{\mathcal{S}} \cap N^c)=0$, which implies $\int_\mathcal{S}m_t(dx)dt = \nu(\mathcal{S}) = \nu(\mathring{\mathcal{S}} \cap N^c) = 0$. Let us show now that $\mu(\mathcal{C})=0$. Using the supermartingale property of $U$,
$$\mathbb E^{\tilde{\mathbb P}}\left[v(\tilde \tau, \tilde X_{\tilde \tau}) + \int_0^{\tilde \tau} f(t, \tilde X_t)dt\right]\leq \mathbb E^{\tilde{\mathbb P}}[v(0, \tilde X_0)]$$
which implies
$$\int_{[0, T]\times \mathbb R}v(t, x)\mu(dt, dx) + \int_0^T \int_{\mathbb R}f(t, x)m_t(dx)dt \leq \int_\mathbb R v(0, x)m_0^*(dx).$$
The above inequality, together with $v \geq g$, leads to
$$\int_{[0, T]\times \mathbb R}(v-g)(t, x, \bar \mu)\mu(dt, dx)=0$$
and henceforth
$$0 = \int_{\mathcal{C}}(v-g)(t, x, \bar \mu)\mu(dt, dx) + \int_{\mathcal{S}}(v-g)(t, x, \bar \mu)\mu(dt, dx)=\int_{\mathcal{C}}(v-g)(t, x, \bar \mu)\mu(dt, dx).$$
Now, since $(t, x)\mapsto (v-g)(t, x, \bar \mu)>0$ on $\mathcal{C}$, we must have $\mu(\mathcal{C})=0$.\vspace{5pt}

\noindent
\textit{Third Step: Uniqueness of the maximizer.} Assume that $(\mu^1, m^1)\in \mathcal{R}$ and $(\mu^2, m^2)\in \mathcal{R}$ are two maximizers. By the previous step, for $i=1, 2$, $\int_{\mathcal{S}}m_t^i(dx) dt = 0$ and $\mu^{i}(\mathcal{C})=0$. By Theorem \ref{exit rep} (the assumption being verified using the same ideas as in Proposition 2 of \cite{chen2011} with some modifications adapted to our framework), there exist, for each $i=1, 2$, a filtered probability space $(\Omega^i, \mathcal{F}^i, \mathbb F^i, \mathbb P^i)$, an $\mathbb F^i$-adapted process $X^i$ and an $\mathbb F^i$-Brownian motion $W^i$ such that
$$X^i_{t}=X^i_0 + \int_0^{t}b(s, X^i_s)ds + \int_0^{t}\sigma(X^i_s)dW^i_s, \quad \mathbb P^i \circ (X_0^i)^{-1}= m_0^*,$$
$$\mu^i=\mathbb P^i\circ\left(\tau_\mathcal{C}^i, X^i_{\tau_\mathcal{C}^i}\right)^{-1},\quad m_t^i(B)= \mathbb E^{\mathbb P^i}\left[ \mathds{1}_B(X_t^i) \mathds{1}_{t< \tau_\mathcal{C}^i}\right],  \quad B\in \mathcal{B}(\mathbb R), \quad t-a.e.,$$
where $\tau_\mathcal{C}^i=\inf\{t\geq 0: (t, X_t^i)\notin \mathcal{C} \}$. By the pathwise uniqueness of the following SDE, 
$$dX_t=b(t, X_t)dt + \sigma(X_t)dW_t,$$
we get the uniqueness in law. This implies that $\mathbb P^1\circ (X^1)^{-1}=\mathbb P^2\circ (X^2)^{-1}=:P$ on $C([0, T])$. For $i=1, 2$, using that $\tau^{i}_\mathcal{C}$ is $\sigma(X^i)$-measurable, there exists a measurable map $\varphi^i:C([0, T])\rightarrow [0, T]$ such that $\tau^{i}_\mathcal{C}=\varphi^i(X^i)$. In particular, for any bounded and measurable function $\psi:C([0, T])\rightarrow \mathbb R$, $\mathbb E^{\mathbb P^1}[\psi(X^1)\varphi^1(X^1)]=\mathbb E^{\mathbb P^2}[\psi(X^2)\varphi^2(X^2)]$, that is 
$$\int_{C([0, T])}\psi(x)[\varphi^1(x)-\varphi^2(x)]P(dx)=0.$$
Taking $\psi=\varphi^1-\varphi^2$ we deduce that $\varphi^1=\varphi^2$ $P$-a.e. This is sufficient to conclude that $m^1=m^2$ and $\mu^1=\mu^2$.
\end{proof}
\begin{remark}
Note that the above theorem gives sufficient conditions which guarantee the representation of the unique best response as a pure solution. Furthermore, the stopping time involved in the probabilistic representation is a Markov stopping time.
\end{remark}

\section{Sufficient conditions for Assumption \ref{main assump SC}}\label{app suff cond SC}

In this section, we provide sufficient conditions on $b$, $\sigma$, $f$, $g$, $m_0^*$ and $\mathcal{O}$ such that Assumption \ref{main assump SC} is verified. We only prove the uniqueness of the best response, since it is immediate  to observe that the other conditions are satisfied.

\begin{assumption}\label{suff conditions assump SC}
We assume the following:
\begin{enumerate}[(1)]
\item $\mathcal{O}$ is a bounded open interval, $A$ is convex and $\sigma=1$.
\item $m_0^*$ admits a bounded density with respect to the Lebesgue measure.
\item $b(t, x, a)=b_1(t, x) + b_2(t, x)a$, with $b_1$ and $b_2$ continuous, Lipschitz in $x$ uniformly on $t$ and with linear growth.
\item For all $t\in [0, T]$, $x\in \bar{\mathcal{O}}$, $\eta\in \mathcal{P}^{sub}(\bar{\mathcal{O}})$, $a\in A$, $f(t, x, \eta, a)= f_1(t, x, \eta) + f_2(t, x, a)$. The function $f_1$ has the form
$$f_1(t, x, \eta) = \hat f(t, x) \bar f\left(t, \int_{\bar{\mathcal{O}}}\hat f(t, y)\eta(dy)\right),$$
where $\hat f$ and $\bar f$ are jointly measurable, bounded and continuous in $x$ for each $t$, and $z\mapsto \bar f(t, z)$ is non-increasing. Moreover, there exists $\bar c_f\geq 0$ such that for all $t\in [0, T]$, $x, x'\in \mathbb R$,
$$|\hat f(t, x)-\hat f(t, x')|\leq \bar c_f|x-x'|,\quad |\hat f(t, x)|\leq \bar c_f(1+|x|),$$
$$|\bar f(t, x)-\bar f(t, x')|\leq \bar c_f|x-x'|, \quad |\bar f(t, x)|\leq \bar c_f(1+|x|).$$
The function $f_2$ is jointly measurable, continuous in $(x, a)$ for each $t$, and for each $(t, x)$, $a\mapsto f_2(t, x, a)$ is strictly concave.
\item $g$ has the form 
$$g(t, x, \mu)= g_1(t, x) g_2\left(\int_{[0, T]\times \bar{\mathcal{O}}}g_1(s, y)\mu(ds, dy)\right) + g_3(t, x),$$
where $g_1$, $g_2$ and $g_3$ are continuous and $g_2$ is non-increasing. Moreover, there exists $\bar c_g\geq 0$ such that for all $t, t'\in [0, T]$, $x, x'\in \mathbb R$,
$$|g_1(t, x)-g_1(t', x')|\leq \bar c_g(|t-t'| + |x-x'|),\quad |g_2(x)-g_2(x')|\leq \bar c_g|x-x'|.$$
Finally,  for a fixed $\mu\in \mathcal{P}([0, T]\times \bar{\mathcal{O}})$, $(t, x)\mapsto g(t, x, \mu)\in C^{1, 2}([0, T]\times \bar{\mathcal{O}})$ and $g(t, x, \mu)=0$ for $(t, x)\in (0, T)\times \partial \mathcal{O}$.
\end{enumerate}
\end{assumption}

Let $(\bar \mu, \bar m)\in \mathcal{R}$. Consider a complete probability space $(\Omega, \mathcal{F}, \mathbb P)$ supporting a Brownian motion $W=(W_t)_{t\in [0, T]}$. For $t\in [0, T]$, we denote by $\mathbb{F}^t$ the filtration given by $\mathcal{F}_{s}^t=\sigma\left(W_r^t, t \leq r \leq s\right) \vee \mathcal{N}$, $s\geq t$. Here $\mathcal{N}$ is the set of $\mathbb P$-null sets and $W^t_s:=W_s-W_t$, $s\geq t$, is the translated Brownian motion. Denote by $\mathcal{A}_t$ the set of $\mathbb{F}^t$-progressively measurable process with values in $A$. Omitting the dependence on $\bar m$ and $\bar \mu$, consider the value function
\begin{align}\label{value_fct SC}
v(t, x)=\sup _{\alpha\in \mathcal{A}_t} \mathbb{E}\left[\int_{t}^{T\wedge\tau_\mathcal{O}} f\left(s, X_{s}^{t, x, \alpha}, \bar m_s^x, \alpha_s\right) ds + g\left(T\wedge\tau_\mathcal{O}, X_{T\wedge\tau_\mathcal{O}}^{t, x, \alpha}, \bar \mu\right)\right].
\end{align}
where $X^{t, x, \alpha}$ is the unique strong solution to the SDE
$$dX^{t, x, \alpha}_s=b(s, X^{t, x, \alpha}_s, \alpha_s)ds + dW_s; \,\,\ X_t^{t, x, \alpha}=x.$$

\vspace{10pt}
\noindent The following theorem is a particular case of Theorems 2.1 and 2.2, Chapter 4, in \cite{bensoussan1982}.

\begin{theorem}\label{theoremstrong}
Let Assumptions \ref{assump existence SC} and \ref{suff conditions assump SC} be satisfied. The value function $v$ is the unique solution belonging to $C([0, T]\times \bar{\mathcal{O}})\cap W^{1, 2, 2}((0, T)\times \mathcal{O})$\footnote{The Sobolev space $W^{1, 2, 2}((0, T)\times \mathcal{O})$ represents the set of functions $u$ such that $u$, $\partial_t u$, $\partial_x u$, $\partial_{xx}u \in L^2((0, T)\times \mathcal{O})$, where the derivatives are understood in the sense of distributions.}, satisfying the following Hamilton-Jacobi-Bellman equation (HJB)
\begin{equation}\label{HJB}
\begin{aligned} 
\frac{\partial v}{\partial t}(t, x)+\sup_{a\in A}\left[\mathcal{L} v(t, x, a) + f(t, x, \bar m_t^x, a)\right]=0,&\quad (t, x) \in(0, T) \times \mathcal{O}, \\ 
v(t, x)=0,&\quad (t, x) \in(0, T)\times \partial \mathcal{O}, \\ 
v(T, x)=g(T, x, \bar\mu),&\quad x \in \mathcal{O}. \end{aligned}
\end{equation}
\end{theorem}

\begin{theorem}\label{unique max sc}
Under Assumptions \ref{assump existence SC} and \ref{suff conditions assump SC}, there is a unique maximizer $(\mu^\star, m^\star)$ of $\Gamma[\bar\mu, \bar m]$ in $\mathcal{R}$.
\end{theorem}

\begin{proof}
\textit{First Step: Optimality implies strict Markovian maximizer.} We first prove that the set of maximizers is contained in the set of measures associated to strict controls. Let $(\mu, m)\in \mathcal{R}$ and consider the transition kernel $(\nu_{t, x})_{t, x}\subset \mathcal{P}(A)$ such that
$$m_t(dx, da)dt=\nu_{t, x}(da)m_t^x(dx)dt.$$
Let $D(A)=\{\delta_a: a\in A\}$ be the set of Dirac masses on $A$ which is in $\mathcal{B}(\mathcal{P}(A))$ since it is closed (recall that $A$ is compact). Consider the Borel set $B=\{(t, x)\in [0, T]\times \bar{\mathcal{O}}: \nu_{t, x}\in D(A)\}$. Assume that the measure $m$ is not associated to a strict control, i.e. $\int_{B^c}m_t^x(dx)dt>0$, and let us show that we can construct a measure $\tilde m$ associated to a strict control which leads to a strictly higher reward. Define $\alpha(t, x):=\int_A a\nu_{t, x}(da)$ and $\tilde m_t(dx, da):=\delta_{\alpha(t, x)}(da)m_t^x(dx)$. The function $\alpha$ is measurable and takes values in $A$ since this set is convex and $\nu_{t,x}(\cdot) \in \mathcal{P}(A)$. Let $u\in C_b^{1, 2}([0, T]\times \bar{\mathcal{O}})$, we obtain
\begin{align*}
&\int_\mathcal{O} u(0, x)m_0^*(dx) + \int_0^T \int_{\bar{\mathcal{O}} \times A} \left(\frac{\partial u}{\partial t} +\mathcal L u\right) (t, x, a)\tilde m_t(dx,da)dt\\
&\quad =\int_\mathcal{O} u(0, x)m_0^*(dx) + \int_0^T \int_{\bar{\mathcal{O}} \times A} \left(\frac{\partial u}{\partial t} +\mathcal L u\right) (t, x, a) m_t(dx,da)dt\\
&\quad = \int_{\Sigma} u(t, x)\mu(dt, dx)
\end{align*}
Therefore, $(\mu, \tilde m)\in \mathcal{R}$. Now, since $a\mapsto f(t, x, a)$ is strictly concave, by Jensen's inequality we get
\begin{align*}
\Gamma[\bar \mu, \bar m](\mu, m)
&= \int_0^T\int_{\bar{\mathcal{O}}}\int_A f(t, x, \bar m_t^x, a)\nu_{t, x}(da)m_t^x(dx)dt + \int_{\Sigma} g(t, x, \bar \mu)\mu(dt, dx) \\
&< \int_0^T\int_{\bar{\mathcal{O}}} f(t, x, \bar m_t^x, \alpha(t, x))m_t^x(dx)dt + \int_{\Sigma} g(t, x, \bar \mu)\mu(dt, dx) \\
&= \Gamma[\bar \mu, \bar m](\mu, \tilde m),
\end{align*}
the inequality being strict since $\text{supp}(\nu_{t, x})$ contains more than one element on $B^c$ and $\int_{B^c}m_t^x(dx)dt>0$. We have shown that for every $(\mu, m)\in \mathcal{R}$ there exists a strict admissible control with corresponding strictly higher reward, henceforth, the set of maximizers is contained on the set of LP solutions with strict control.\vspace{5pt}

\noindent
\textit{Second Step: Uniqueness of the Markovian strict control.} Let $(\mu, m)$ be a maximizer. By Step 1, we can write
$$m_t(dx, da)=\delta_{\alpha(t, x)}(da)m^{x}_t(dx),$$
for some measurable functions $\alpha$. Using Theorem \ref{theoremstrong} and applying an analogue proof as in Theorem 2.29 in \cite{dlt2021} we get
$\int_0^T\int_{\bar{\mathcal{O}}\times A} F_v(t, x, \bar m_t^x, a)m_t(dx, da)dt=0$, where 
$$F_v(t, x, \bar m_t^x, a):=\frac{\partial v}{\partial t}(t, x)+(\mathcal{L} v)(t, x, a) + f(t, x, \bar m_t^x, a).$$ 
From the HJB equation, a.e. for all $a\in A$, we get $F_v(t, x, \bar m_t^x, a)\leq 0$. We deduce that $m_t^x(dx)dt$-a.e., 
$$\alpha(t, x)\in \argmax_{a\in A} F_v(t, x, \bar m_t^x, a).$$
Since for each $(t, x)\in [0, T]\times \bar{\mathcal{O}}$, $b(t, x, \cdot)$ is affine and $f(t, x, \bar m_t^x, \cdot)$ is strictly concave, there exists a unique maximizer $\alpha^\star(t, x)\in A$ of
\begin{equation}\label{hamiltonian}
A\ni a\mapsto F_v(t, x, \bar m_t^x, a).
\end{equation}
Therefore we get $\alpha(t, x)=\alpha^\star(t, x)$ $m^x_t(dx)dt$-a.e. Without loss of generality we can assume $\alpha=\alpha^\star$ because $\alpha^\star$ is measurable (it is a particular case of Theorem 18.19 in \cite{aliprantis2007}). \vspace{5pt}

\noindent \textit{Third Step: Uniqueness.} If $(\mu^1, m^1)$ and $(\mu^2, m^2)$ are two maximizers, by the second step,
$$m_t^k(dx, da)dt=\delta_{\alpha^\star(t, x)}m_t^{x, k}(dx)dt, \quad k=1, 2.$$
Now, by Theorem C.6 in \cite{dlt2021}, for $k=1, 2$, there exist a filtered probability space $(\Omega^k, \mathcal{F}^k, \mathbb F^k, \mathbb P^k)$, an $\mathbb F^k$-adapted process $X^k$, an $\mathbb F^k$-Brownian motion $W^k$ such that
$$X_{t\wedge \tau^k_\mathcal{O}}^k= X_0^k + \int_0^{t\wedge \tau^k_\mathcal{O}}b(s, X_s^k,\alpha^\star(s, X_s^k))ds + W_{t\wedge \tau^k_\mathcal{O}}^k, \quad \mathbb P^k \circ (X_0^k)^{-1}= m_0^*,$$
$$\mu^k =\mathbb P^k \circ (T\wedge \tau_\mathcal{O}^k, X_{T\wedge \tau_\mathcal{O}^k}^k)^{-1},$$
$$m_t^k(B\times C)= \mathbb E^{\mathbb P^k}\left[ \mathds{1}_B(X_t^k)\mathds{1}_{\alpha^\star(t, X_t^k)}(C) \mathds{1}_{t< T\wedge \tau_\mathcal{O}^k}\right],  \quad B\in \mathcal{B}(\bar{\mathcal{O}}), \quad C\in \mathcal{B}(A), \quad t-a.e.$$ 
By a similar proof as in Chapter 4, Proposition 3.10 in \cite{karatzas1998a}, we get
$$\mathbb P^1\circ \left(X_{\cdot\wedge \tau^1_\mathcal{O}}^1, T\wedge \tau^1_\mathcal{O}\right)^{-1}=\mathbb P^2\circ \left(X_{\cdot\wedge \tau^2_\mathcal{O}}^2, T\wedge \tau^2_\mathcal{O}\right)^{-1},$$
in $C([0, T])\times [0, T]$. As a consequence $\mu^1=\mu^2$ and $m^1=m^2$.
\end{proof}

\begin{corollary}[\textit{Pure solution representation of the best response}]\label{corollary pure control}
Under the Assumptions of Theorem \ref{unique max sc}, the unique best response $(\mu^\star, m^\star)$ can be represented as \textit{a pure solution}, i.e.$$\mu^\star=\mathbb P\circ \left(X_{T\wedge \tau_\mathcal{O}^X}, T\wedge \tau_\mathcal{O}^X\right)^{-1},$$
$$m_t^\star(B\times C)= \mathbb E^{\mathbb P}\left[ \mathds{1}_B(X_t)\mathds{1}_{\alpha^\star(t, X_t)}(C) \mathds{1}_{t< T\wedge \tau_\mathcal{O}^X}\right],  \quad B\in \mathcal{B}(\bar{\mathcal{O}}), \quad C\in \mathcal{B}(A), \quad t-a.e.$$
where $(\Omega, \mathcal{F}, \mathbb{P}, W)$ represents the initial probabilistic set up (see pg. 34), $\alpha^\star$ is the unique maximizer of the Hamiltonian \eqref{hamiltonian} and $X$ is the strong solution of the SDE associated to $\alpha^\star$. 
\end{corollary}

\section{Technical lemmas}

In this section we give analogous results to Appendix F in \cite{dlt2021} in the case of test functions with polynomial growth.

\begin{lemma}\label{slutsky_p}
Let $\mathcal{X}$ and $\mathcal{Y}$ be complete separable metric spaces and let $\varphi:\mathcal{X}\times \mathcal{Y}\rightarrow\mathbb R$ be continuous and satisfying the following growth condition: there exist $c\geq 0$ and $(x_0, y_0)\in \mathcal{X}\times \mathcal{Y}$ such that for all $(x, y)\in \mathcal{X}\times \mathcal{Y}$
$$|\varphi(x, y)|\leq c (1 + d_\mathcal{X}(x, x_0)^p + d_\mathcal{Y}(y, y_0)^p).$$
Consider a sequence $(\nu^n)_{n\geq 1}\in \mathcal{M}_p(\mathcal{X})$ converging to $\nu\in \mathcal{M}_p(\mathcal{X})$ in $\tau_p$ such that there exists $C>0$ so that
$$\sup_{n\geq 1}\int_\mathcal{X}(1+d_\mathcal{X}(x, x_0)^p)\nu^n(dx)\leq C.$$
Consider also a sequence $(y^n)_{n\geq 1}\in \mathcal{Y}$ converging to $y\in \mathcal{Y}$ such that there exists a compact set $\mathcal{K}\subset\mathcal{Y}$ so that for all $n\geq 1$, $y^n\in \mathcal{K}$. Then, 
  $$\int_{\mathcal{X}}\varphi(x, y^n)\nu^n(dx)\underset{n\rightarrow\infty}{\longrightarrow} \int_{\mathcal{X}}\varphi(x, y)\nu(dx).$$
\end{lemma}

The next Lemma is related to Lemma A.3 of \cite{lacker2015}.

\begin{lemma}[Stable convergence: the $p$-growth case]\label{slutsky_stable_p}
Let $\Theta$, $\mathcal{X}$, $\mathcal{Y}$ be complete, separable metric spaces. Let $\eta\in \mathcal{M}_p(\Theta)$. Let $\varphi:\Theta \times \mathcal{X}\times \mathcal{Y}\rightarrow\mathbb R$, be a measurable map and assume that for every $t\in \Theta$, $\varphi(t, \cdot)$ is continuous. We assume the following growth condition on $\varphi$: there exists $c\geq 0$ and $(t_0, x_0, y_0)\in \Theta\times \mathcal{X}\times \mathcal{Y}$
$$|\varphi(t, x, y)|\leq c (1+d_\Theta(t, t_0)^p + d_\mathcal{X}(x, x_0)^p + d_\mathcal{Y}(y, y_0)^p).$$
Suppose that a sequence of measurable functions $\psi^n:\Theta\rightarrow \mathcal{Y}$ converges $\eta$-a.e. in $\Theta$ to a measurable function $\psi:\Theta\rightarrow \mathcal{Y}$ and that $(\nu^n_t(dx)\eta(dt))_{n\geq 1}\subset \mathcal{M}_p(\Theta\times \mathcal{X})$ converges to $\nu_t(dx)\eta(dt)\in \mathcal{M}_p(\Theta\times \mathcal{X})$ in $\bar \tau_p$, where $(\nu^n)_{n\geq 1}$ and $\nu$ are transition kernels from $\Theta$ to $\mathcal{X}$. Suppose also that there exists a constant $C>0$ such that $\eta$-a.e. 
$$\sup_{n\geq 1}\int_\mathcal{X}(1+d_\mathcal{X}(x, x_0)^p)\nu_t^n(dx)\leq C.$$ 
Moreover, suppose that there exists a compact set $\mathcal{K}\subset \mathcal{Y}$ such that for all $n\geq 1$, $\psi^n(t)\in \mathcal{K}$ $\eta$-a.e. Then,
$$\int_\Theta \int_{\mathcal{X}}\varphi(t, x, \psi^n(t))\nu^n_t(dx)\eta(dt)\underset{n\rightarrow\infty}{\longrightarrow} \int_\Theta \int_{\mathcal{X}}\varphi(t, x, \psi(t))\nu_t(dx)\eta(dt).$$
\end{lemma}

\printbibliography

\end{document}